\documentclass[11pt]{article}

\usepackage[a4paper,margin=28mm]{geometry}
\usepackage[round,authoryear]{natbib}

\usepackage{microtype}
\usepackage{amsmath,amssymb,amsthm,mathtools,bm}
\usepackage{booktabs,array,tabularx}
\usepackage{graphicx}
\usepackage{float}
\usepackage{xcolor}
\usepackage[colorlinks=true,allcolors=blue!55!black]{hyperref}
\usepackage[nameinlink,noabbrev]{cleveref}

\usepackage{authblk}
\usepackage{enumitem}

\newfloat{algorithm}{tbp}{loa}
\floatname{algorithm}{Algorithm}

\newcommand{\R}{\mathbb{R}}
\newcommand{\E}{\mathbb{E}}
\newcommand{\Pp}{\mathbb{P}}
\newcommand{\PS}{\mathbb{P}_{S}}
\newcommand{\ES}{\mathbb{E}_{S}}
\newcommand{\Pstar}{\mathbb{P}^{*}}
\newcommand{\Estar}{\mathbb{E}^{*}}
\newcommand{\norm}[1]{\left\lVert #1\right\rVert}
\newcommand{\op}{\mathrm{op}}
\newcommand{\diag}{\operatorname{diag}}
\newcommand{\tr}{\operatorname{tr}}
\newcommand{\rank}{\operatorname{rank}}
\newcommand{\argmin}{\operatorname*{arg\,min}}
\newcommand{\dto}{\rightsquigarrow}
\newcommand{\pto}{\xrightarrow{p}}

\newcommand{\appref}[1]{\hyperref[#1]{Appendix~\ref*{#1}}}
\newcommand{\bbeta}{\bm{\beta}}
\newcommand{\bbetahat}{\widehat{\bm{\beta}}}
\newcommand{\bbetastar}{\widehat{\bm{\beta}}^{*}}
\newcommand{\by}{\bm{y}}
\newcommand{\br}{\bm{r}}
\newcommand{\bg}{\bm{g}}
\newcommand{\bs}{\bm{s}}

\newcommand{\bz}{\bm{z}}
\newcommand{\bw}{\bm{w}}
\newcommand{\bzero}{\bm{0}}

\newtheorem{assumption}{Assumption}
\newtheorem{theorem}{Theorem}
\newtheorem{proposition}{Proposition}
\newtheorem{corollary}{Corollary}
\newtheorem{lemma}{Lemma}
\theoremstyle{remark}
\newtheorem{remark}{Remark}
\crefname{assumption}{Assumption}{Assumptions}
\Crefname{assumption}{Assumption}{Assumptions}
\crefname{proposition}{Proposition}{Propositions}
\Crefname{proposition}{Proposition}{Propositions}
\crefname{corollary}{Corollary}{Corollaries}
\Crefname{corollary}{Corollary}{Corollaries}
\crefname{lemma}{Lemma}{Lemmas}
\Crefname{lemma}{Lemma}{Lemmas}

\title{Bootstrap Error Estimation and Sketch-Size Selection for Sketched Ridge Regression}

\author[1]{Akito Narahara}
\author[1]{Takayuki Kawashima}
\author[1,2]{Takafumi Kanamori}

\affil[1]{Department of Mathematical and Computing Science, \newline
Institute of Science Tokyo, Tokyo, Japan}
\affil[2]{RIKEN Center for Advanced Intelligence Project, Tokyo, Japan}

\date{}

\hypersetup{
  pdftitle={Bootstrap Error Estimation and Sketch-Size Selection for Sketched Ridge Regression},
  pdfauthor={Akito Narahara, Takayuki Kawashima, Takafumi Kanamori},
  pdfkeywords={randomized numerical linear algebra, ridge regression, bootstrap, Monte Carlo error, tolerance limits, sketch-size selection}
}

\begin{document}
\maketitle

\begin{abstract}
Randomized sketching reduces the computational cost of large ridge-regression problems, but the coefficient error depends on the realized sketch. We extend the paired-row bootstrap from randomized least squares to ridge regression, enabling coefficient-error estimation using only compressed data. Under a fixed coefficient dimension and increasing data and sketch sizes, we derive asymptotic linear representations and Gaussian limits for the sketched estimator and the conditional bootstrap distribution. These results establish uniform consistency of the bootstrap error distribution and asymptotically exact coverage when the estimator and error bound are computed from the same sketch. For sketches with independent, mean-zero, variance-one entries, an explicit covariance formula separates the effects of the residual, regularization, and fourth moment of the sketch entries, and shows that Rademacher entries minimize the leading covariance matrix in the Loewner order. We also develop a fast linearized bootstrap, an order-statistic correction for finitely many bootstrap replicates, and a Bonferroni rule for selecting from a fixed set of sketch sizes. Experiments on two real and two synthetic data sets support the proposed methods. With a sketch size 15 times the number of coefficients, 199 bootstrap replicates, and nominal coverage of 95\%, the bootstrap with refitting attains coverage 
between 92.0\% and 95.3\%; after the order-statistic correction, coverage ranges from 94.7\% to 97.7\%.
\end{abstract}


\section{Introduction}
\label{sec:introduction}

Randomized numerical linear algebra (RandNLA) uses randomization to reduce a
large linear algebra problem to a smaller problem that is cheaper to solve.
These methods are widely used in scientific computing, statistics, and
machine learning; broad accounts include \citet{mahoney2011randomized},
\citet{martinsson2020randomized}, \citet{murray2023perspective}, and
\citet{derezinski2024recent}. A basic example is \emph{sketch-and-solve}. A
tall design matrix and its response vector are compressed to fewer rows, and
the reduced problem is solved in place of the original problem. The resulting
coefficient vector depends on the random sketch. Its approximation error is
therefore random even when the data are fixed.

This paper studies how to estimate that error after computing a classical
sketched ridge-regression solution. Prior work includes error bounds derived
before the sketch is drawn, analyses of statistical risk, iterative methods,
model averaging, and lower bounds
\citep{avron2017sharper,wang2017sketched,chowdhury2018iterative,
kacham2022sketching,patil2024asymptotically,lejeune2024pseudoinverse}.
We instead estimate the coefficient error for the realized sketch. We address
three questions:
\begin{enumerate}[label=\textbf{Q\arabic*.}]
  \item How accurately can the coefficient error at the current sketch size be estimated from the compressed problem?
  \item How should the reported error bound account for Monte Carlo error from a finite number of bootstrap replicates?
  \item How can a sketch size be selected from a finite set of candidates while controlling the probability that the coefficient error at the selected size exceeds a prescribed tolerance?
\end{enumerate}
The design-response pair is treated as fixed throughout, and probability refers to randomized computation and bootstrap resampling.

The closest related method is the bootstrap of compressed design--response
pairs developed by \citet{lopes2018least} for randomized least squares. Their
method estimates error from the compressed rows after the sketch has been
computed. They establish asymptotically valid coverage for an error bound
computed from the realized sketch as the number of bootstrap replicates
diverges. They also allow a general error norm and use the $m^{-1/2}$ error
rate to guide sketch-size selection. Thus, resampling compressed rows,
estimating error after sketching, proving asymptotic bootstrap validity, and
using the $m^{-1/2}$ error rate for sketch-size selection are not new
contributions of this paper.

Building on \citet{lopes2018least}, we extend their validity theory to
sketched ridge regression and allow the normalized regularization parameter
to vary with problem size. This extension provides the theoretical basis for
the results below. Our main contributions beyond this extension are as
follows.
\begin{enumerate}
 \item \textbf{Asymptotic covariance for ridge regression.}
Within this ridge framework, we derive the covariance matrix of the leading
error term and separate the contributions from the residuals, regularization,
and fourth moment of the sketch entries. Among sketches whose entries are
i.i.d. with mean zero and variance one, Rademacher sketches yield the smallest
covariance matrix in the Loewner order.

  \item \textbf{Fast linearized bootstrap.}
  We replace repeated ridge refits by a first-order update. For each fixed
  bootstrap replicate, the resulting linearized error is first-order
  equivalent to the error obtained by refitting; see
  \Cref{prop:linearized-equivalence}.

  \item \textbf{Correction for finitely many bootstrap replicates and adjustment for multiple comparisons.}
  A distribution-free order-statistic correction controls the probability of
  underestimating the exact conditional bootstrap quantile when only $B$
  replicates are used. Separately, a Bonferroni correction computes an
  adjusted error bound at each candidate sketch size. If the error bound at
  every candidate size has the required marginal coverage, this rule
  asymptotically controls the probability of selecting a sketch size whose
  coefficient error exceeds the tolerance; see \Cref{prop:mc-protected} and
  \cref{thm:simultaneous-selection}.

  \item \textbf{Numerical evaluation.}
  The experiments examine four issues: coverage when the estimator and error
  bound use the same sketch; the accuracy and cost of linearization; the
  order-statistic correction for Monte Carlo error; and the effect of a
  multiple-comparison adjustment on the selected sketch size and the
  probability of exceeding the prescribed tolerance. Earlier benchmark
  settings and additional sensitivity analyses are reported in
  \appref{sec:supp-experiments}.
\end{enumerate}

Our problem differs from prediction-risk tuning for sketched ensembles \citep{patil2024asymptotically}, runtime-oriented solver tuning \citep{cho2025surrogate}, sequential procedures that increase the sketch size \citep{chen2025sequential}, and adaptive selection of random-subspace dimension in nonlinear least squares \citep{bellavia2026variable}. We report an error bound for one sketched ridge estimate, conditional on the full data being fixed.

Our analysis assumes exact arithmetic. Floating-point stability is a separate implementation issue. Uncertainty-aware stopping rules and stability analyses for randomized least-squares solvers provide relevant tools for combining error estimation with reliable linear algebra \citep{pritchard2023uncertainty,epperly2024stable,xu2025refinement,jia2026numerical}.

\section{Related work}
\label{sec:related-work}

\subsection{Randomized algorithms for ridge regression and sketch-size selection}

\citet{lu2013faster} used structured random projections to accelerate ridge
regression. For regularized least-squares problems, sketch-size bounds can
depend on the statistical dimension rather than the matrix rank
\citep{avron2017sharper}. \citet{wang2017sketched} compared the classical
sketch and the Hessian sketch from optimization and statistical perspectives,
while \citet{chowdhury2018iterative,kacham2022sketching} studied iterative
algorithms and lower bounds. Subsequent work developed the RidgeSketch solver
\citep{gazagnadou2022ridgesketch}, sketched Krylov methods for computing a
regularization path \citep{wang2023krylov}, surrogate models for automatic
solver tuning \citep{cho2025surrogate}, sequential estimators with increasing
sketch sizes \citep{chen2025sequential}, and adaptive choices of random-subspace
dimension for nonlinear least squares \citep{bellavia2026variable}. This
literature focuses mainly on approximation accuracy, statistical risk,
computational cost, and convergence of iterative solvers. Our goal is
complementary: after a particular sketch has been drawn and the sketched ridge
problem has been solved, we estimate the Euclidean error in the coefficient
vector and use such error bounds to select a sketch size from a set of
candidates.

\subsection{Bootstrap error estimation and statistical inference under sketching}

The closest related method is the bootstrap of \citet{lopes2018least} for
randomized least squares. Their method resamples the rows of the compressed
design-response pair and uses the resulting bootstrap distribution to estimate
the error of a randomized solution without returning to the full data. Their
theory establishes bootstrap consistency and asymptotically valid coverage for
a general error norm when the problem size, sketch size, and number of
bootstrap replicates diverge. They also use the $m^{-1/2}$ error rate to
extrapolate an error estimate from a pilot sketch to larger sketch sizes. Thus,
row-wise resampling, error estimation after sketching,
asymptotic bootstrap validity, and bootstrap-based guidance for choosing the
sketch size are already present in \citet{lopes2018least}.

We extend this framework to ridge regression. We formulate the bootstrap using
the ridge normal equations, allow the normalized regularization parameter to
vary with the problem size, and derive the covariance matrix of the leading
error term. The covariance formula shows that, among sketches whose entries
are i.i.d. with mean zero and variance one, the Rademacher sketch has the
smallest first-order covariance matrix in the Loewner order. We also introduce a linearized
bootstrap and prove that, for any fixed bootstrap replicate, it is first-order
equivalent to refitting the ridge model. The Gaussian limit, bootstrap consistency,
and coverage results extend the least-squares theory of
\citet{lopes2018least} to ridge regression; they do not introduce a new
resampling principle.

We also address two issues not studied by \citet{lopes2018least}. First, when
only finitely many bootstrap replicates are available, a one-sided
order-statistic correction controls the probability that the adjusted error
bound falls below the exact conditional bootstrap quantile. Second, when
several candidate sketch sizes are considered, a Bonferroni correction accounts for multiple
comparisons. Unlike the $m^{-1/2}$ extrapolation of \citet{lopes2018least}, our
selection rule in \cref{eq:selection-rule} computes a separate
Bonferroni-adjusted error bound at each candidate size. If the error bound at
each candidate size has the required marginal coverage,
the rule asymptotically controls the probability of selecting a sketch size
whose coefficient error exceeds the prescribed tolerance.

Bootstrap methods for estimating error due to randomized computation have
also been developed for randomized matrix multiplication, sketched singular
value decomposition, randomized Newton methods, Oja's algorithm, random
Fourier features, and operator-norm errors in covariance estimation and
sketching
\citep{lopes2019matrix,lopes2020svd,chen2020newton,lunde2021oja,
yao2023rff,lopes2023operator}. Other work uses leave-one-out and jackknife
methods to estimate error and output variance in randomized matrix computations
\citep{epperly2024efficient}.

A separate literature studies statistical inference after sketching. Some
analyses account for both sampling variation in the data and randomness from
the sketch \citep{ahfock2021statistical,chi2022projector,lee2022heteroskedastic}.
Others condition on the full data and treat the sketch as the source of
randomness \citep{browne2024inference,wang2024quadratic,zhang2025framework}.
Our setting is of the second type. The data are fixed, and the quantity of
interest is the Euclidean coefficient error produced by one randomized
sketch. Online covariance estimation for sketched Newton methods is another approach to
quantifying uncertainty from randomized computation
\citep{kuang2026covariance,wang2026accelerated}.

\subsection{Scope of the theory}

The theory assumes that the coefficient dimension is fixed and that the rows
of the sketching matrix are i.i.d. These assumptions are important because
the bootstrap resamples the compressed design--response rows as independent
units drawn from a common distribution. The validity of the bootstrap can change when the coefficient
dimension grows proportionally with the sample size
\citep{clarte2024bootstrap}. The present theory also does not directly cover
structured transforms such as the subsampled randomized Hadamard transform
(SRHT) and CountSketch; extending the method to these structured sketches would require
resampling procedures that preserve their dependence structure. Finally, our
analysis assumes exact arithmetic and does not account for floating-point
error. It is therefore
complementary to forward- and backward-stability analyses of randomized
least-squares algorithms \citep{epperly2024stable,xu2025refinement}.

\section{Problem formulation and sketching methods}
\label{sec:method}

Throughout the paper, $\norm{\bm x}_2$ denotes the Euclidean norm of a
vector $\bm x$, while $\norm{M}_{\mathrm F}$ and $\norm{M}_{\op}$ denote
the Frobenius and operator norms of a matrix $M$, respectively, with
$\norm{M}_{\op}:=\sup_{\norm{\bm x}_2=1}\norm{M\bm x}_2$.
For an event $\mathcal A$, let $\mathbf1\{\mathcal A\}$ denote its
indicator. We write $I_k$ for the $k\times k$ identity matrix,
$\bm e_j\in\R^k$ for its $j$th column,
$\diag(\bm a)$ for the diagonal matrix with diagonal vector
$\bm a\in\R^k$, and $\bm a\odot\bm b$ for the componentwise product of
$\bm a,\bm b\in\R^k$. For a square matrix $N$, let $\tr(N)$ and
$\rank(N)$ denote its trace and rank. For symmetric matrices $N$ and
$L$, let $\lambda_{\min}(N)$ denote the smallest eigenvalue of $N$, and
write $N\succeq L$ when $\bm x^\top(N-L)\bm x\geq0$ for every vector
$\bm x$ of the appropriate dimension.
Finally, $\Pp$ denotes the underlying probability measure,
$\E\bm Y:=\int\bm Y\,d\Pp$ for an integrable random vector $\bm Y$, and
$\operatorname{Var}(\bm Y):=
\E[(\bm Y-\E\bm Y)(\bm Y-\E\bm Y)^\top]$
when $\bm Y$ is square-integrable.

\subsection{Ridge objective and sketched estimator}
\label{sec:normalization}
For each positive integer $n$, let $X_n\in\R^{n\times d}$ and
$\by_n\in\R^n$ be deterministic. The coefficient dimension $d$ is fixed in
the main theory. For a regularization parameter $\lambda_n>0$ that may depend
on $n$, define
\begin{equation}
  \bbeta_n
  :=
  \argmin_{\bbeta\in\R^d}
  \left\{
    \frac1n\norm{X_n\bbeta-\by_n}_2^2
    +
    \lambda_n\norm{\bbeta}_2^2
  \right\}.
  \label{eq:full-ridge}
\end{equation}
Write
\begin{equation}
  H_n
  :=
  \frac1nX_n^\top X_n,
  \qquad
  \bg_n
  :=
  \frac1nX_n^\top\by_n,
  \qquad
  A_n
  :=
  (H_n+\lambda_n I_d)^{-1}.
  \label{eq:full-moments}
\end{equation}
The normal equation for \cref{eq:full-ridge} gives
\[
  (H_n+\lambda_n I_d)\bbeta_n=\bg_n,
  \qquad
  \bbeta_n=A_n\bg_n.
\]

Because $\lambda_n>0$, the coefficient matrix in every full-data, sketched,
and bootstrap ridge normal equation below is positive definite. The limiting regularization parameter in
\cref{ass:regime} may equal zero.

We use the normalized criterion in \cref{eq:full-ridge} throughout.
An implementation based on the unnormalized criterion
$\norm{X_n\bbeta-\by_n}_2^2+\rho_n\norm{\bbeta}_2^2$
produces the same estimator after the conversion
$\lambda_n=\rho_n/n$. This conversion is used only to interpret the
archived benchmark settings reported in \appref{sec:source-experiments}.

\subsection{Independent-row sketches}
\label{sec:sketches}

For a positive integer $m$, let $\bs_{n,1},\ldots,\bs_{n,m}\in\R^n$
be independent and identically distributed random vectors. Let
$\mathcal S_{n,m}:=\sigma(\bs_{n,1},\ldots,\bs_{n,m})$ be the
$\sigma$-algebra generated by the sketch rows. For
$\mathcal A\in\mathcal S_{n,m}$ and every integrable
$\mathcal S_{n,m}$-measurable random variable $Z$, write
\[
  \PS(\mathcal A):=\Pp(\mathcal A),
  \qquad
  \ES Z:=\E Z.
\]
Their common distribution is isotropic:
\[
  \ES(\bs_{n,i}\bs_{n,i}^\top)=I_n,
  \qquad i=1,\ldots,m.
\]
Set
\begin{equation}
  S_{n,m}=\frac1{\sqrt m}
  \begin{bmatrix}
    \bs_{n,1}^\top\\[-1mm]
    \vdots\\[-1mm]
    \bs_{n,m}^\top
  \end{bmatrix}\in\R^{m\times n}.
  \label{eq:sketch-matrix}
\end{equation}
For each sketch row, define the corresponding compressed design vector and
response scalar by
\begin{equation}
  \bz_{n,i}:=\frac1{\sqrt n}X_n^\top\bs_{n,i}\in\R^d,
  \qquad
  u_{n,i}:=\frac1{\sqrt n}\by_n^\top\bs_{n,i}\in\R.
  \label{eq:compressed-rows}
\end{equation}
Then
\begin{equation}
  \widehat H_{n,m}:=\frac1m\sum_{i=1}^m\bz_{n,i}\bz_{n,i}^\top,
  \qquad
  \widehat{\bg}_{n,m}:=\frac1m\sum_{i=1}^m\bz_{n,i}u_{n,i},
  \label{eq:sketched-moments}
\end{equation}
satisfy
\[
  \ES(\widehat H_{n,m})=H_n,
  \qquad
  \ES(\widehat{\bg}_{n,m})=\bg_n,
\]
and the sketched estimator is
\begin{equation}
  \bbetahat_{n,m}
  :=(\widehat H_{n,m}+\lambda_n I_d)^{-1}\widehat{\bg}_{n,m}.
  \label{eq:sketched-ridge}
\end{equation}
Equivalently, \cref{eq:sketched-ridge} minimizes
$n^{-1}\norm{S_{n,m}(X_n\bbeta-\by_n)}_2^2+\lambda_n\norm{\bbeta}_2^2$.
The factor $1/n$ makes the sketched moments unbiased for the full-data moments.
In particular,
\begin{equation}
  \widehat H_{n,m}-H_n
  =\frac1m\sum_{i=1}^m
  \frac1nX_n^\top(\bs_{n,i}\bs_{n,i}^\top-I_n)X_n.
  \label{eq:sketched-gram-average}
\end{equation}
Thus, randomness enters through sample averages of fixed-dimensional
quantities, after which the ridge normal equations are solved.

The following three common sketching methods illustrate this setup.
\begin{description}
\item[Independent-entry projections.] The coordinates of $\bs_{n,i}$ are
i.i.d. with mean zero and variance one. Gaussian and Rademacher projections
are the main examples. Applying a dense projection can be computationally
expensive. The compressed design--response pairs are independent across
sketch rows, so the ordinary nonparametric bootstrap can resample these pairs
directly.
\item[Row sampling.] Let $J_i$ be a random index with
$\Pp(J_i=j)=p_{n,j}>0$ and set
$\bs_{n,i}=\bm e_{J_i}/\sqrt{p_{n,J_i}}$. Then
\cref{eq:compressed-rows} selects an original design row and its corresponding
response, with inverse-probability scaling. Uniform row sampling uses
$p_{n,j}=1/n$.
\item[Structured sketches.] SRHT and CountSketch can be faster to apply than
dense projections. Their rows are not generally independent within a single
sketch. A bootstrap method for these sketches must preserve this dependence
structure.
\end{description}

For $\alpha\in(0,1)$, define the exact $(1-\alpha)$-quantile of the coefficient error
\begin{equation}
  q_{n,m}(1-\alpha)
  :=\inf\left\{t\geq0:
  \PS\!\left(
  \norm{\bbetahat_{n,m}-\bbeta_n}_2\leq t
  \right)\geq1-\alpha\right\}.
  \label{eq:coverage-target}
\end{equation}
Our goal is to estimate \cref{eq:coverage-target} from one realized sketch
and report the sketched coefficient vector together with an estimated error
bound.

\section{Bootstrap error estimation and sketch-size selection}
\label{sec:bootstrap-calibration}

\subsection{Bootstrap of compressed design-response pairs with ridge refitting}
\label{sec:exact-bootstrap}

Each sketch row produces one compressed design-response pair
$(\bz_{n,i},u_{n,i})$. The bootstrap resamples these $m$ pairs with
replacement, recomputes the two sketched moments, and solves the ridge problem
again. Thus, every bootstrap replicate requires a new ridge fit.
Coverage is established asymptotically in \cref{thm:exact-coverage}.
Recall $\mathcal S_{n,m}$ from \cref{sec:sketches}. For every event
$\mathcal A$ and every integrable random variable $Z$ involving both sketch
and bootstrap randomness, define
\[
  \Pstar(\mathcal A):=\Pp(\mathcal A\mid\mathcal S_{n,m}),
  \qquad
  \Estar Z:=\E(Z\mid\mathcal S_{n,m}).
\]

For a positive integer $B$, conditional on the observed compressed rows, draw
$B$ independent multinomial weight vectors
\begin{equation}
  \bw_b=(w_{b,1},\ldots,w_{b,m})^\top
  \sim\operatorname{Multinomial}
  \left(m;\frac1m,\ldots,\frac1m\right),
  \qquad b=1,\ldots,B.
  \label{eq:multinomial-weights}
\end{equation}
Define the weighted moments
\begin{equation}
  \widehat H_{n,m,b}^{*}:=\frac1m\sum_{i=1}^m w_{b,i}\bz_{n,i}\bz_{n,i}^\top,
  \qquad
  \widehat{\bg}_{n,m,b}^{*}:=\frac1m\sum_{i=1}^m w_{b,i}\bz_{n,i}u_{n,i},
  \label{eq:weighted-moments}
\end{equation}
and the bootstrap estimator obtained by refitting
\begin{equation}
  \bbetastar_{n,m,b}:=(\widehat H_{n,m,b}^{*}+\lambda_n I_d)^{-1}
  \widehat{\bg}_{n,m,b}^{*}.
  \label{eq:exact-bootstrap-estimator}
\end{equation}
For a distribution function $F$, define its $p$-quantile by the generalized
inverse $F^{-1}(p):=\inf\{t:F(t)\geq p\}$. Define the bootstrap errors and
their empirical $(1-\alpha)$-quantile by
\begin{equation}
  e_{n,m,b}^{*}:=\norm{\bbetastar_{n,m,b}-\bbetahat_{n,m}}_2,
  \qquad
  \widehat\epsilon_{n,m,B}(\alpha)
  :=\operatorname{quantile}_{1-\alpha}
  \{e_{n,m,1}^{*},\ldots,e_{n,m,B}^{*}\}.
  \label{eq:exact-bootstrap-threshold}
\end{equation}
Here, $\operatorname{quantile}_p\{\cdots\}$ denotes the empirical
$p$-quantile of the listed values under the generalized-inverse convention.
We use $\widehat\epsilon_{n,m,B}(\alpha)$ as the empirical bootstrap error
bound.

\begin{algorithm}[H]
\caption{Bootstrap of compressed design-response pairs with ridge refitting}
\label{alg:exact-bootstrap}
\noindent\textbf{Input:} $X_n,\by_n,m,\lambda_n,\alpha,B$.
\begin{enumerate}[label=\arabic*.,leftmargin=*,itemsep=2pt,topsep=3pt]
\item Generate $S_{n,m}$ and compute the compressed design--response pairs
$(\bz_{n,i},u_{n,i})_{i=1}^m$.
\item Compute $\bbetahat_{n,m}$ from \cref{eq:sketched-ridge}.
\item For $b=1,\ldots,B$:
  \begin{enumerate}[label=(\alph*),leftmargin=2em,itemsep=1pt,topsep=1pt]
  \item Draw $\bw_b$ from \cref{eq:multinomial-weights}.
  \item Compute $\bbetastar_{n,m,b}$ from
  \cref{eq:weighted-moments,eq:exact-bootstrap-estimator}.
  \item Set
  $e_{n,m,b}^{*}=\norm{\bbetastar_{n,m,b}-\bbetahat_{n,m}}_2$.
  \end{enumerate}
\end{enumerate}
\noindent\textbf{Output:}
$\bigl(\bbetahat_{n,m},\widehat\epsilon_{n,m,B}(\alpha)\bigr)$.
\end{algorithm}

After compression, the method uses only $m$ design--response pairs. A direct
implementation costs $O(md^2+d^3)$ for the initial fit and
$O\{B(md^2+d^3)\}$ for the bootstrap fits, excluding the cost of applying the
sketch.

\subsection{Linearized bootstrap}
\label{sec:linearized-bootstrap}

Let
\begin{equation}
  \widehat A_{n,m}:=(\widehat H_{n,m}+\lambda_n I_d)^{-1},
  \qquad
  \widehat r_{n,m,i}:=u_{n,i}-\bz_{n,i}^\top\bbetahat_{n,m}.
  \label{eq:sketched-residuals}
\end{equation}
We linearize the map
$(H,\bm g)\mapsto(H+\lambda_n I_d)^{-1}\bm g$ at the observed sketched
moments $(\widehat H_{n,m},\widehat{\bg}_{n,m})$. For a perturbation
$(V,\bm v)\in\R^{d\times d}\times\R^d$, the derivative at this point is
\begin{equation}
  (V,\bm v)\longmapsto
  \widehat A_{n,m}(\bm v-V\bbetahat_{n,m}).
  \label{eq:realized-derivative}
\end{equation}
The bootstrap perturbations are
\begin{equation}
\begin{aligned}
  \widehat H_{n,m,b}^{*}-\widehat H_{n,m}
  &=\frac1m\sum_{i=1}^m(w_{b,i}-1)\bz_{n,i}\bz_{n,i}^\top,\\
  \widehat{\bg}_{n,m,b}^{*}-\widehat{\bg}_{n,m}
  &=\frac1m\sum_{i=1}^m(w_{b,i}-1)\bz_{n,i}u_{n,i}.
\end{aligned}
  \label{eq:bootstrap-moment-perturbations}
\end{equation}
Substitution of \cref{eq:bootstrap-moment-perturbations} into the derivative in \cref{eq:realized-derivative} gives
\begin{align}
  \widetilde\Delta_{n,m,b}^{*}
  &:=\widehat A_{n,m}\left\{
  (\widehat{\bg}_{n,m,b}^{*}-\widehat{\bg}_{n,m})
  -(\widehat H_{n,m,b}^{*}-\widehat H_{n,m})\bbetahat_{n,m}
  \right\}\notag\\
  &=\widehat A_{n,m}\frac1m\sum_{i=1}^m
  (w_{b,i}-1)\bz_{n,i}\widehat r_{n,m,i}.
  \label{eq:linearized-delta}
\end{align}
Thus $\widetilde\Delta_{n,m,b}^{*}$ is the first-order approximation to
$\bbetastar_{n,m,b}-\bbetahat_{n,m}$, the bootstrap perturbation obtained by
refitting after the observed compressed rows are reweighted. The corresponding
empirical influence vectors are
\begin{equation}
  \widehat{\bm\psi}_{n,m,i}
  :=\widehat A_{n,m}
  \left\{\bz_{n,i}\widehat r_{n,m,i}
  -\lambda_n\bbetahat_{n,m}\right\},
  \qquad i=1,\ldots,m.
  \label{eq:empirical-influence}
\end{equation}
The sketched normal equation gives
$m^{-1}\sum_i\widehat{\bm\psi}_{n,m,i}=\bzero$. Because
$\sum_i(w_{b,i}-1)=0$, \cref{eq:linearized-delta} can also be written as
\begin{equation}
  \widetilde\Delta_{n,m,b}^{*}
  =\frac1m\sum_{i=1}^m(w_{b,i}-1)\widehat{\bm\psi}_{n,m,i}.
  \label{eq:linearized-as-influence}
\end{equation}
Define
\begin{equation}
  \widetilde e_{n,m,b}^{*}:=\norm{\widetilde\Delta_{n,m,b}^{*}}_2,
  \qquad
  \widetilde\epsilon_{n,m,B}(\alpha)
  :=\operatorname{quantile}_{1-\alpha}
  \{\widetilde e_{n,m,1}^{*},\ldots,\widetilde e_{n,m,B}^{*}\}.
  \label{eq:linearized-error}
\end{equation}
Then $\widetilde e_{n,m,b}^{*}$ approximates the bootstrap error obtained by
refitting, $e_{n,m,b}^{*}$. The quantity
$\widetilde\epsilon_{n,m,B}(\alpha)$ is the corresponding empirical error
bound from the linearized bootstrap. The $d\times m$ matrix with
columns $\widehat{\bm\psi}_{n,m,1},\ldots,\widehat{\bm\psi}_{n,m,m}$
can be formed once. All $B$ replicates can then be computed by matrix
multiplication, at an additional cost of $O(Bmd)$.

\subsection{Correcting Monte Carlo error from finitely many bootstrap replicates}
\label{sec:mc-protection}

Conditional on the observed sketch, let $p:=1-\alpha$ and define the exact
conditional bootstrap quantile
\[
  q_{n,m}^{*}(p)
  :=\inf\left\{t\geq0:
  \Pstar\!\left(e_{n,m,1}^{*}\leq t\right)\geq p\right\}.
\]
This is the $p$-quantile that would be available if the conditional bootstrap
distribution were known exactly. The empirical bootstrap quantile in
\cref{eq:exact-bootstrap-threshold} estimates $q_{n,m}^{*}(p)$ from only $B$
replicates and can therefore be too small because of Monte Carlo variation in
the upper tail.

To control the probability of underestimating this quantile, we replace
the unadjusted empirical quantile by a suitably high order statistic.
Fix in advance the number of replicates $B$ and an underestimation probability
$\delta\in(0,1)$, and order the observed bootstrap errors as
$e_{n,m,(1)}^{*}\leq\cdots\leq e_{n,m,(B)}^{*}$. By the definition of
$q_{n,m}^{*}(p)$, the conditional probability that one bootstrap error is
strictly below this exact quantile is at most $p$. Hence the number of
bootstrap errors that fall strictly below $q_{n,m}^{*}(p)$ is stochastically no larger
than a binomial random variable with parameters $B$ and $p$.  We therefore
choose the smallest order-statistic rank
\begin{equation}
  k_B(p,\delta)
  :=\min\left\{k\in\{1,\ldots,B\}:
  \Pp\{\operatorname{Bin}(B,p)\leq k-1\}\geq1-\delta\right\},
  \label{eq:mc-order}
\end{equation}
where $\operatorname{Bin}(B,p)$ denotes a binomial random variable. If the
set in \cref{eq:mc-order} is empty, put $k_B(p,\delta):=\infty$. The
error bound adjusted by the order-statistic correction is
\begin{equation}
  \widehat\epsilon_{n,m,B}^{\mathrm{MC}}(\alpha,\delta)
  :=
  \begin{cases}
    e_{n,m,(k_B(p,\delta))}^{*},&k_B(p,\delta)<\infty,\\
    +\infty,&k_B(p,\delta)=\infty.
  \end{cases}
  \label{eq:mc-threshold}
\end{equation}
Thus \cref{eq:mc-threshold} replaces the unadjusted empirical quantile by a
sufficiently high order statistic. Whenever $k_B(p,\delta)<\infty$, the
adjusted error bound is no smaller than the exact conditional bootstrap
quantile $q_{n,m}^{*}(p)$ with conditional probability at least $1-\delta$.
This is a distribution-free upper tolerance bound for the exact conditional
bootstrap quantile; see \Cref{prop:mc-protected} and
\citet{wilks1941determination}. The adjustment limits the conditional
probability of Monte Carlo underestimation to $\delta$. It does not account
for the difference between the bootstrap distribution and the
coefficient-error distribution at a finite sketch size. If no order statistic
provides the requested underestimation probability, $+\infty$ indicates that
no finite-valued error bound with this distribution-free guarantee is
available. A finite rank in \cref{eq:mc-order} exists exactly when
$p^B\leq\delta$.

\subsection{Selecting a sketch size from a finite grid to meet an error tolerance}
\label{sec:certified-selection}

Suppose that a practitioner considers a prespecified increasing grid of
candidate sketch sizes $m_1<\cdots<m_K$. For candidate $j$, let $B_j$ be the
number of bootstrap replicates, let $\bbetahat_{n,m_j}$ be the corresponding
sketched estimator, and let $\widehat\epsilon_{n,m_j,B_j}(\gamma)$ be the
empirical bootstrap error bound in \cref{eq:exact-bootstrap-threshold} with
nominal noncoverage probability $\gamma$. Candidate $j$ satisfies the
selection criterion for a prescribed tolerance $\tau>0$ when its error bound
is at most $\tau$.

To account for comparing $K$ candidate sketch sizes, we use the Bonferroni
correction and set the nominal noncoverage probability to $\alpha/K$ at each
size. Candidate $j$ therefore satisfies the selection criterion when
$\widehat\epsilon_{n,m_j,B_j}(\alpha/K)\leq\tau$. We select the smallest
candidate satisfying this condition:
\begin{equation}
  \widehat j
  :=\min\left\{j\in\{1,\ldots,K\}:
  \widehat\epsilon_{n,m_j,B_j}(\alpha/K)\leq\tau\right\},
  \label{eq:selection-rule}
\end{equation}
when this set is nonempty. If the set is empty, put $\widehat j:=0$; this
means that no candidate satisfies the selection criterion and no error
guarantee is reported.
The candidate sketches may be independent, or they may be nested prefixes of
one larger sketch with i.i.d. rows. If the error bound at every candidate size
has the required marginal coverage, division of $\alpha$ by $K$ gives the
asymptotic familywise error control in \Cref{thm:simultaneous-selection}.
Independence among the candidates is not required.

Computing \cref{eq:selection-rule} requires a separate bootstrap error bound at
every candidate size. A cheaper preliminary estimate can be obtained from a
single pilot size $m_0$ by using the approximate $m^{-1/2}$ error scaling.
Specifically,
if
\[
  \widehat\epsilon_{n,m,B}(\alpha)
  \approx
  \widehat\epsilon_{n,m_0,B}(\alpha)\sqrt{\frac{m_0}{m}},
\]
then solving the approximate requirement
$\widehat\epsilon_{n,m,B}(\alpha)\leq\tau$ for $m$ gives
\begin{equation}
  \widehat m_{\mathrm{pilot}}
  :=\left\lceil m_0
  \left\{\frac{\widehat\epsilon_{n,m_0,B}(\alpha)}{\tau}\right\}^2
  \right\rceil.
  \label{eq:adaptive-m}
\end{equation}
\Cref{eq:adaptive-m} is therefore an extrapolated sketch-size
estimate based on one pilot run and the $m^{-1/2}$ approximation.
The familywise error-control result in \Cref{thm:simultaneous-selection}
applies only to the finite-grid rule in \cref{eq:selection-rule}; it does not
apply to the pilot-based extrapolation in \cref{eq:adaptive-m}.

\section{Asymptotic theory}
\label{sec:theory}

Every asymptotic statement is taken along one sequence indexed by $n\to\infty$.  The data $(X_n,\by_n)$ are deterministic, $d$ is fixed, and the sketch size is $m_n\to\infty$.  There is no separate limit with $n$ fixed and $m\to\infty$, and no rate restriction between $m_n$ and $n$ is imposed beyond the assumptions below.  A sequence $B_n\to\infty$ is required only when an empirical bootstrap distribution or empirical bootstrap quantile is used.

Recall the definitions of $\PS$, $\ES$, and $\mathcal S_{n,m}$ from
\cref{sec:sketches}.  At sketch size $m_n$, define
\[
  \mathcal S_n:=\mathcal S_{n,m_n},
  \qquad
  \Pstar(\mathcal A):=\Pp(\mathcal A\mid\mathcal S_n),
  \qquad
  \Estar Z:=\E(Z\mid\mathcal S_n).
\]
In coverage statements involving $B_n$ bootstrap draws, $\Pp$ denotes
the joint law of the sketch and those draws.  We use $O_p^*$ and $o_p^*$
for stochastic orders under the conditional bootstrap law, and we use
conditional weak convergence in $\PS$-probability. Precise definitions are
given in \appref{app:proofs}; all conditional statements condition on the
observed sketch $\mathcal S_n$.
At sketch size $m_n$,
\begin{equation}
\begin{aligned}
  \widehat H_{n,m_n}-H_n
  &=\frac1{m_n}\sum_{i=1}^{m_n}
    (\bz_{n,i}\bz_{n,i}^{\top}-H_n),\\
  \widehat\bg_{n,m_n}-\bg_n
  &=\frac1{m_n}\sum_{i=1}^{m_n}
    (\bz_{n,i}u_{n,i}-\bg_n).
\end{aligned}
\label{eq:moment-vector}
\end{equation}

\begin{assumption}[Asymptotic regime and fixed-design limits]
\label{ass:regime}
The regularization parameters satisfy $\lambda_n>0$ for every $n$ and,
as $n\to\infty$,
\[
  \lambda_n\to\lambda_\infty\geq0,
  \qquad H_n\to H,
  \qquad \bg_n\to\bg,
\]
and there exists $c_0>0$ such that
$\lambda_{\min}(H_n+\lambda_n I_d)\geq c_0$ for all sufficiently large $n$.
Consequently,
$A_n\to A:=(H+\lambda_\infty I_d)^{-1}$ and
$\bbeta_n\to\bbeta:=A\bg$.
\end{assumption}

Define the mean-zero influence vector for compressed row $i$ by
\begin{equation}
  \bm\psi_{n,i}
  :=A_n\{\bz_{n,i}(u_{n,i}-\bz_{n,i}^{\top}\bbeta_n)
       -\lambda_n\bbeta_n\}, 
  \label{eq:derivative-map}
\end{equation}
where $\br_n:=\by_n-X_n\bbeta_n$. 
Using \cref{eq:compressed-rows}, the expression in braces can be written in
terms of the original sketch row:
\begin{align}
  \bz_{n,i}(u_{n,i}-\bz_{n,i}^{\top}\bbeta_n)-\lambda_n\bbeta_n
  &=\frac1nX_n^\top\bs_{n,i}\bs_{n,i}^\top\br_n
    -\frac1nX_n^\top\br_n\notag\\
  &=\frac1nX_n^\top(\bs_{n,i}\bs_{n,i}^\top-I_n)\br_n.
  \label{eq:influence-identity}
\end{align}
The ridge normal equation gives
$n^{-1}X_n^\top\br_n=\lambda_n\bbeta_n$. Therefore,
\begin{equation}
  \bm\psi_{n,i}
  =A_n\frac1nX_n^\top
   (\bs_{n,i}\bs_{n,i}^{\top}-I_n)\br_n,
  \label{eq:single-influence}
\end{equation}
and $\ES\bm\psi_{n,i}=\bzero$.

\begin{assumption}[Sketch-row regularity]
\label{ass:sketch}
For every $n$, $\bs_{n,1},\ldots,\bs_{n,m_n}$ are i.i.d. and isotropic:
$\ES(\bs_{n,1}\bs_{n,1}^{\top})=I_n$.  Moreover,
\begin{equation}
  \sup_n\ES\left[
    \norm{\bz_{n,1}\bz_{n,1}^{\top}-H_n}_{\mathrm F}^{4}
    +\norm{\bz_{n,1}u_{n,1}-\bg_n}_2^{4}
  \right]<\infty.
  \label{eq:design-regularity}
\end{equation}
Finally, assume that there exists a finite matrix $\Omega\in\R^{d\times d}$ such that
\begin{equation}
  \Omega_n
  :=\ES\left[
  (\bm\psi_{n,1}-\ES\bm\psi_{n,1})
  (\bm\psi_{n,1}-\ES\bm\psi_{n,1})^\top
  \right]
  =\ES(\bm\psi_{n,1}\bm\psi_{n,1}^\top)
  \longrightarrow\Omega,
  \label{eq:omega}
\end{equation}
where the second equality uses $\ES\bm\psi_{n,1}=\bzero$.
\end{assumption}

\begin{assumption}[Nondegeneracy]
\label{ass:design}
$\rank(\Omega)\geq1$.
\end{assumption}

\begin{lemma}[Sufficient conditions for sketch-row regularity]
\label{lem:concrete-conditions}
Under \cref{ass:regime}, suppose
$\bs_{n,1},\ldots,\bs_{n,m_n}$ are independent and that the coordinates of
each row are i.i.d. copies of a scalar random variable $s$ satisfying
$\E s=0$, $\E s^2=1$, and $\E|s|^8<\infty$. Assume that
$n^{-1}\by_n^{\top}\by_n$ converges to a finite limit and that, for every
fixed symmetric $C\in\R^{d\times d}$ and every fixed $\bm c\in\R^d$,
\begin{equation}
  \frac1{n^2}\sum_{j=1}^n
  \left[\bm e_j^{\top}
  \{X_nCX_n^{\top}+X_n\bm c\,\by_n^{\top}\}
  \bm e_j\right]^2
  \label{eq:diagonal-condition}
\end{equation}
converges to a finite limit.  Then \cref{ass:sketch} holds.
\end{lemma}
A proof is given in \appref{app:sufficient-conditions}.

The results are ordered by their logical dependencies. \Cref{prop:covariance}
gives the asymptotic linear expansion and its covariance, the following remark
compares sketches with i.i.d. standardized entries, and
\cref{prop:linearized-equivalence} shows that the linearized bootstrap
approximates the bootstrap with refitting.
\Cref{thm:bootstrap-consistency,cor:bootstrap-distribution-consistency,thm:exact-coverage}
then establish Gaussian limits, consistency of the bootstrap distribution, and
coverage for an error bound computed from the observed sketch. The
final results address Monte Carlo error from finitely many bootstrap replicates
and selection among finitely many sketch sizes.

\subsection{Linearization, covariance, and the linearized bootstrap}
\label{sec:covariance}

\begin{proposition}[Asymptotic linear representation and covariance]
\label{prop:covariance}
Under \cref{ass:regime,ass:sketch},
\begin{equation}
  \sqrt{m_n}(\bbetahat_{n,m_n}-\bbeta_n)
  =\frac1{\sqrt{m_n}}\sum_{i=1}^{m_n}\bm\psi_{n,i}+o_p(1).
  \label{eq:target-linearization}
\end{equation}
The covariance matrix of $\bm\psi_{n,i}$ is $\Omega_n$. Because the sketch
rows are i.i.d., the leading sum
$m_n^{-1/2}\sum_{i=1}^{m_n}\bm\psi_{n,i}$ also has covariance matrix
$\Omega_n$. By \cref{ass:sketch}, $\Omega_n\to\Omega$.
If the coordinates of $\bs_{n,i}$ are independent copies of a scalar $s$ satisfying
\begin{equation}
  \E s=0,\qquad \E s^2=1,\qquad \E|s|^4<\infty,
  \label{eq:sketch-moments}
\end{equation}
then, with $\kappa_s:=\E s^4-3$,
\begin{equation}
  \Omega_n
  =A_n\left\{
  \frac{\norm{\br_n}_2^2}{n}H_n
  +\lambda_n^2\bbeta_n\bbeta_n^{\top}
  +\frac{\kappa_s}{n^2}
  X_n^{\top}\diag(\br_n\odot\br_n)X_n
  \right\}A_n^{\top}.
  \label{eq:gamma}
\end{equation}
\end{proposition}
A proof is given in \appref{app:covariance-proof}.

\begin{remark}[Minimum first-order covariance for Rademacher sketches]
\label{cor:rademacher-optimal}
For sketches with i.i.d. standardized entries, Rademacher entries
yield the smallest first-order covariance matrix in the Loewner order.
Indeed, Jensen's inequality gives
\[
  \E s^4
  \geq
  (\E s^2)^2
  =
  1.
\]
Equality holds if and only if $s^2=1$ almost surely; together with
$\E s=0$, this is exactly the Rademacher law with
$\Pp(s=1)=\Pp(s=-1)=1/2$. Let $\Omega_n^{\mathrm{Rad}}$ denote the
covariance in \cref{eq:gamma} under this law. Subtracting the
Rademacher version of \cref{eq:gamma} from the general version gives
\begin{equation}
  \Omega_n-\Omega_n^{\mathrm{Rad}}
  =
  \frac{\E s^4-1}{n^2}
  A_nX_n^\top
  \diag(\br_n\odot\br_n)
  X_nA_n^\top
  \succeq0.
  \label{eq:rademacher-order}
\end{equation}
Thus, Rademacher entries yield the smallest first-order covariance matrix in
the Loewner order among sketches with i.i.d. standardized entries. This
comparison applies only within this class; sparse and structured sketches
require separate analysis.
\end{remark}

Recall that $\widetilde\Delta_{n,m,b}^{*}$ in
\cref{eq:linearized-delta} is the first-order approximation to the
bootstrap perturbation obtained by refitting,
$\bbetastar_{n,m,b}-\bbetahat_{n,m}$.
\begin{proposition}[First-order equivalence of refitting and linearization]
\label{prop:linearized-equivalence}
Under \cref{ass:regime,ass:sketch}, for the first bootstrap replicate
(and hence for any fixed replicate index) and every $\varepsilon>0$,
\begin{equation}
  \Pstar\!\left\{
  \sqrt{m_n}\norm{(\bbetastar_{n,m_n,1}-\bbetahat_{n,m_n})
  -\widetilde\Delta_{n,m_n,1}^{*}}_2>\varepsilon
  \right\}\pto0.
  \label{eq:linearized-equivalence}
\end{equation}
Thus, for any fixed replicate index, the bootstrap perturbations computed by
refitting and by linearization differ by $o_p^*(m_n^{-1/2})$ in
$\PS$-probability.
This statement is not uniform over a growing number of replicates.
\end{proposition}
A proof is given in \appref{app:linearized-proof}.

\subsection{Bootstrap consistency and coverage}
\label{sec:bootstrap-theory}

Using the first bootstrap replicate obtained by refitting, as defined in
\cref{eq:multinomial-weights,eq:weighted-moments,eq:exact-bootstrap-estimator},
define
\begin{equation}
\begin{aligned}
  F_{n,m_n}(t)
  &:=\PS\!\left\{\sqrt{m_n}\norm{\bbetahat_{n,m_n}-\bbeta_n}_2\leq t\right\},\\
  \widehat F_{n,m_n}(t)
  &:=\Pstar\!\left\{\sqrt{m_n}\norm{\bbetastar_{n,m_n,1}-\bbetahat_{n,m_n}}_2\leq t\right\}.
\end{aligned}
\label{eq:ideal-distribution-functions}
\end{equation}

\begin{theorem}[Gaussian limits for the estimator and the conditional bootstrap distribution]
\label{thm:bootstrap-consistency}
Under \cref{ass:regime,ass:sketch,ass:design},
\begin{align}
  \sqrt{m_n}(\bbetahat_{n,m_n}-\bbeta_n)&\dto N(\bzero,\Omega),
  \label{eq:target-clt}\\
  \sqrt{m_n}(\bbetastar_{n,m_n,1}-\bbetahat_{n,m_n})&\dto N(\bzero,\Omega)
  \quad\text{conditionally in $\PS$-probability}.
  \label{eq:bootstrap-clt}
\end{align}
\end{theorem}
A proof is given in \appref{app:bootstrap-proof}.

For $B_n$ bootstrap draws, let
\begin{equation}
  \widehat F_{n,m_n,B_n}(t)
  :=\frac1{B_n}\sum_{b=1}^{B_n}
  \mathbf 1\!\left\{\sqrt{m_n}\,e_{n,m_n,b}^{*}\leq t\right\}.
  \label{eq:empirical-bootstrap-cdf}
\end{equation}

\begin{corollary}[Uniform consistency of the bootstrap distribution]
\label{cor:bootstrap-distribution-consistency}
Under the assumptions of \cref{thm:bootstrap-consistency},
\begin{equation}
  \sup_{t\in\R}|F_{n,m_n}(t)-\widehat F_{n,m_n}(t)|\pto0.
  \label{eq:ideal-uniform}
\end{equation}
If, in addition, $B_n\to\infty$, then
\begin{equation}
  \sup_{t\in\R}|F_{n,m_n}(t)-\widehat F_{n,m_n,B_n}(t)|\pto0.
  \label{eq:empirical-uniform}
\end{equation}
\end{corollary}
A proof is given in \appref{app:distribution-proof}.

\begin{theorem}[Asymptotically exact coverage for error bounds computed from the observed sketch]
\label{thm:exact-coverage}
Let $B_n\to\infty$.  Under \cref{ass:regime,ass:sketch,ass:design}, for every $\alpha\in(0,1)$, let
$\widehat\epsilon_{n,m_n,B_n}(\alpha)$ and
$\widetilde\epsilon_{n,m_n,B_n}(\alpha)$ be the empirical error bounds
defined in \cref{eq:exact-bootstrap-threshold,eq:linearized-error} with
$m=m_n$ and $B=B_n$.  Then
\begin{align}
  \Pp\left\{
  \norm{\bbetahat_{n,m_n}-\bbeta_n}_2
  \leq\widehat\epsilon_{n,m_n,B_n}(\alpha)
  \right\}&\longrightarrow1-\alpha,
  \label{eq:exact-coverage}\\
  \Pp\left\{
  \norm{\bbetahat_{n,m_n}-\bbeta_n}_2
  \leq\widetilde\epsilon_{n,m_n,B_n}(\alpha)
  \right\}&\longrightarrow1-\alpha.
  \label{eq:linearized-coverage}
\end{align}
Here $\Pp$ is the joint probability over the sketch and the bootstrap draws
used to construct the error bound. The sketched estimator and each error bound
are computed from the same observed sketch.
\end{theorem}
The proof in \appref{app:coverage-proof} combines quantile consistency with
upper and lower bounds for the random bootstrap error bound.

\subsection{Order-statistic correction and sketch-size selection}

\begin{proposition}[One-sided order-statistic correction for a conditional bootstrap quantile]
\label{prop:mc-protected}
Fix positive integers $n,m,B$ and levels $\alpha,\delta\in(0,1)$, and
put $p:=1-\alpha$.  Recall from \cref{sec:mc-protection} that
$q_{n,m}^{*}(p)$ is the conditional $p$-quantile of one bootstrap error
obtained by refitting. Then
\begin{equation}
  \Pstar\left\{
  \widehat\epsilon_{n,m,B}^{\mathrm{MC}}(\alpha,\delta)
  \geq q_{n,m}^{*}(1-\alpha)
  \right\}\geq1-\delta
  \label{eq:mc-conditional}
\end{equation}
almost surely with respect to the sketch law $\PS$. Thus, conditional
on the observed sketch, the adjusted error bound in \cref{eq:mc-threshold}
is no smaller than the exact conditional bootstrap quantile with probability
at least $1-\delta$.  If $m_n\to\infty$ under
\cref{ass:regime,ass:sketch,ass:design}, then
\begin{equation}
  \liminf_{n\to\infty}
  \Pp\left\{
  \norm{\bbetahat_{n,m_n}-\bbeta_n}_2
  \leq\widehat\epsilon_{n,m_n,B_n}^{\mathrm{MC}}(\alpha,\delta)
  \right\}\geq1-\alpha-\delta
  \label{eq:mc-coverage}
\end{equation}
for any positive-integer sequence $B_n$ for which the adjusted order
statistic is finite.
\end{proposition}
The same conditional result also holds for the linearized errors and their
corresponding exact conditional quantile,
because the linearized replicate errors are conditionally i.i.d.
A proof is given in \appref{app:mc-proof}.

The rule in \cref{eq:selection-rule} is defined for one fixed problem.
For its asymptotic analysis, fix a positive integer $K$,
a level $\alpha\in(0,1)$, and a tolerance $\tau>0$. For each $n$, let
\[
  m_{n,1}<\cdots<m_{n,K}
\]
be the ordered candidate sizes. At candidate $j$, write
\[
  \bbetahat_{n,j}:=\bbetahat_{n,m_{n,j}}
\]
for the sketched estimator, and let
$\widehat\epsilon_{n,j}(\alpha/K)$ be a nonnegative error bound for that
estimator.
Here $j$ indexes the position in the candidate grid, and $n$ indexes the
asymptotic sequence. 

Candidate $j$ satisfies the selection criterion when
$\widehat\epsilon_{n,j}(\alpha/K)\leq\tau$. As in
\cref{eq:selection-rule}, select the smallest candidate satisfying this
criterion:
\[
  \widehat j_n
  :=
  \begin{cases}
    \displaystyle
    \min\left\{j\in\{1,\ldots,K\}:
    \widehat\epsilon_{n,j}(\alpha/K)\leq\tau\right\},
    &\text{if the set is nonempty},\\[2mm]
    0,&\text{if the set is empty}.
  \end{cases}
\]
Thus $\widehat j_n=j\geq1$ selects candidate size $m_{n,j}$.
The value $\widehat j_n=0$ means that no candidate satisfies the selection
criterion. The event to be controlled is the selection of a candidate whose
coefficient error exceeds $\tau$.

\begin{theorem}[Familywise error control for selection among sketch sizes]
\label{thm:simultaneous-selection}
Under the preceding setup, suppose that $m_{n,j}\to\infty$ for each
$j=1,\ldots,K$ and that the error bound for each candidate has asymptotic
coverage at least $1-\alpha/K$:
\[
  \liminf_{n\to\infty}
  \Pp\left\{
  \norm{\bbetahat_{n,j}-\bbeta_n}_2
  \leq\widehat\epsilon_{n,j}(\alpha/K)
  \right\}
  \geq1-\frac{\alpha}{K},
  \qquad j=1,\ldots,K.
\]
Then, regardless of dependence among the candidate sketches,
\begin{equation}
  \limsup_{n\to\infty}
  \Pp\bigg(
  \bigcup_{j=1}^K
  \left\{
  \widehat j_n=j,\,
  \norm{\bbetahat_{n,j}-\bbeta_n}_2>\tau
  \right\}
  \bigg)
  \leq\alpha.
  \label{eq:familywise-coverage}
\end{equation}
\end{theorem}

A sufficient condition for the required marginal coverage is to apply
\cref{thm:exact-coverage} separately to each candidate sequence, with
$B_{n,j}\to\infty$ for every fixed $j$. Using nominal coverage
$1-\alpha/K$ with fixed $B_{n,j}$ does not establish this condition. With a
fixed number of bootstrap replicates, the theorem applies only if a
finite-sample procedure has been shown to provide the required coverage at
every candidate size. The unadjusted empirical bootstrap quantile alone has no
such guarantee.

Equivalently, the asymptotic probability of selecting an estimator whose
coefficient error exceeds the prescribed tolerance is at most $\alpha$.
A proof is given in \appref{app:selection-proof}.

\section{Numerical experiments}
\label{sec:experiments}

We organize the experiments around four separate questions.
\begin{enumerate}
\item Does the reported error bound cover the coefficient error at the requested nominal level?
\item How close is the linearized bootstrap to the bootstrap with refitting, and how much computation does linearization save?
\item When the number of bootstrap replicates is finite, does the order-statistic correction reduce the probability that the reported error bound is too small?
\item How does an adjustment for multiple comparisons across candidate sketch sizes affect the selected sketch size?
\end{enumerate}
Each subsection below addresses one of these questions.

\subsection{Experimental setup and methods}
\label{sec:experimental-design}

The main experiments use two real data sets and two synthetic designs. Their dimensions are listed in \cref{tab:benchmark-design}. RAND-HIE is the RAND Health Insurance Experiment subset distributed with \texttt{statsmodels}; its response is the number of outpatient physician visits \citep{cameron2005microeconometrics,seabold2010statsmodels}. Diabetes is the regression data set distributed with \texttt{scikit-learn}. For both real data sets, the response and all nonconstant feature columns are standardized using population standard deviations, after which an intercept is added.

For each synthetic design, Gaussian matrices are orthonormalized to obtain $Q\in\R^{n\times d}$ and $V\in\R^{d\times d}$. The design matrix is
\[
  X=Q\operatorname{diag}(\sigma_1,\ldots,\sigma_d)V^\top,
\]
where the singular values are geometrically spaced. The response is
\[
  \by=X\bbeta_0+0.5\bm\varepsilon,
  \qquad
  \bm\varepsilon\sim N(\bzero,I_n),
\]
where $\bbeta_0$ is a fixed sinusoidal coefficient vector.
More precisely, for a target condition number $\kappa$, the implementation uses
$\sigma_j=\sqrt n\,\kappa^{-(j-1)/\{2(d-1)\}}$ and
$\beta_{0,j}=c\sin\{0.25+(j-1)(2.75\pi-0.25)/(d-1)\}$, where $c$ is
chosen so that $\norm{\bbeta_0}_2=\sqrt d$.

\begin{table}[H]
\centering
\caption{Data sets used in the main numerical experiments. The condition numbers refer to $X^\top X$.}
\label{tab:benchmark-design}
\begingroup\small\setlength{\tabcolsep}{5pt}
\begin{tabular}{lrrll}
\toprule
Data set & $n$ & $d$ & Type & Condition number \\
\midrule
RAND-HIE & 20190 & 10 & real & -- \\
Diabetes & 442 & 11 & real & -- \\
Ill-conditioned synthetic & 10000 & 20 & synthetic & $10^4$ \\
Well-conditioned synthetic & 10000 & 20 & synthetic & $10^2$ \\
\bottomrule
\end{tabular}
\endgroup
\end{table}

For an observed sketch, the coefficient error is
$\norm{\bbetahat_{n,m}-\bbeta_n}_2$.
The bootstrap error bound is $\widehat\epsilon$, as defined in
\cref{sec:bootstrap-calibration}. In each Monte Carlo repetition, we generate
a new sketch, compute the sketched estimator and its error bound, and record
whether
\[
  \norm{\bbetahat_{n,m}-\bbeta_n}_2
  \leq \widehat\epsilon.
\]
Empirical coverage is the proportion of Monte Carlo repetitions for which this inequality holds.

\Cref{tab:experimental-methods} summarizes the methods.

\medskip
\noindent\textbf{Components of the bootstrap methods.}
The bootstrap methods involve two choices: how to compute replicate errors
and how to choose an error bound from those errors. Replicate errors are
computed either by refitting the ridge model or by linearization. The
order-statistic correction can be applied after either calculation. The
Gaussian approximation is a separate method that does not use bootstrap
resampling.

\begin{table}[t]
\centering
\caption{Methods used in the numerical experiments. The order-statistic correction is defined in \cref{sec:mc-protection}.}
\label{tab:experimental-methods}
\begingroup\small\setlength{\tabcolsep}{4pt}
\begin{tabularx}{\linewidth}{>{\raggedright\arraybackslash}p{0.28\linewidth}>{\raggedright\arraybackslash}X>{\raggedright\arraybackslash}X}
\toprule
Method & How the coefficient-error distribution is approximated & How the error bound is chosen \\
\midrule
Bootstrap with refitting
& Resample the paired compressed rows and solve the ridge problem again in every bootstrap replicate.
& Use the empirical $(1-\alpha)$ quantile of the resulting errors. \\
\addlinespace[0.5em]
Bootstrap with refitting and order-statistic correction
& Use the same errors obtained by refitting as in the preceding row.
& Replace the empirical quantile by the higher order statistic in \cref{eq:mc-threshold}. \\
\addlinespace[0.5em]
Linearized bootstrap
& Use the first-order change in the ridge solution, without solving a new ridge problem in every replicate.
& Use the empirical $(1-\alpha)$ quantile of the linearized errors. \\
\addlinespace[0.5em]
Linearized bootstrap with order-statistic correction
& Use the same linearized errors as in the preceding row.
& Apply the same higher-order-statistic correction. \\
\addlinespace[0.5em]
Gaussian approximation
& Estimate the influence covariance from the observed compressed rows; no bootstrap resampling is used.
& Use the $(1-\alpha)$ quantile of the norm of the corresponding centered Gaussian vector. \\
\bottomrule
\end{tabularx}
\endgroup
\end{table}

For evaluation only, we estimate the coefficient-error quantile using
independent sketches. At each sketch size in the main coverage experiment,
the \emph{independent-sketch reference quantile} is the empirical
$(1-\alpha)$ quantile of coefficient errors from $3000$ newly generated
sketches. It is used only as
an evaluation benchmark.

Unless stated otherwise, the experiments use the normalized regularization parameter
$\lambda=0.1$, independent Gaussian sketch rows, sketch ratios
$m/d\in\{5,10,15,20,25\}$, and $B=199$ bootstrap replicates. Coverage is
estimated from $300$ Monte Carlo repetitions. All results using the
order-statistic correction set $\delta=0.05$.

The Gaussian approximation uses the empirical influence covariance
\[
  \widehat\Omega_{n,m}^{\mathrm{emp}}
  :=\frac1m\sum_{i=1}^m
  \widehat{\bm\psi}_{n,m,i}\widehat{\bm\psi}_{n,m,i}^{\top}
\]
and reports $m^{-1/2}$ times the $(1-\alpha)$ quantile of the norm of a
centered Gaussian vector with this covariance.

In the main coverage experiment, each Gaussian-approximation error bound is
estimated from $2000$ Gaussian draws. Coverage error bars are $95\%$ Wilson
intervals.

For Gaussian sketches, the code samples each compressed row directly from its
known $(d+1)$-dimensional Gaussian distribution instead of explicitly forming
the dense sketch matrix. These two implementations produce the same
distribution.

Quantiles used to compute bootstrap error bounds and independent-sketch
reference quantiles follow the generalized-inverse convention in
\cref{eq:exact-bootstrap-threshold}. Descriptive medians and $90$th
percentiles across repeated runs, including those in
\cref{tab:linearized-summary}, use linearly interpolated sample percentiles.
These descriptive summaries are not used as error bounds.

\subsection{Coverage of the reported error bound}
\label{sec:coverage-results}

The first question is whether the reported bound attains the nominal coverage
$0.95$. We assess two quantities. \emph{Empirical coverage} is the proportion
of repetitions in which the reported bound is at least as large as the actual
coefficient error. The \emph{relative size of the bound} is the mean reported
bound divided by the independent-sketch reference quantile. A ratio near one
means that the average bound is close to this reference $0.95$ quantile.

\Cref{fig:coverage} compares the bootstrap with refitting, the Gaussian
approximation based on the observed sketch, and the bootstrap with refitting
followed by the order-statistic correction.

Across most sketch sizes, the bootstrap with refitting has coverage close to the
nominal level. Some undercoverage remains when the sketch is small. The
Gaussian approximation generally gives a slightly smaller error bound and
greater undercoverage at the smallest sketch ratios. The order-statistic
correction increases the bound obtained by refitting and produces more
conservative coverage.

\begin{figure}[!t]
\centering
\includegraphics[width=\linewidth]{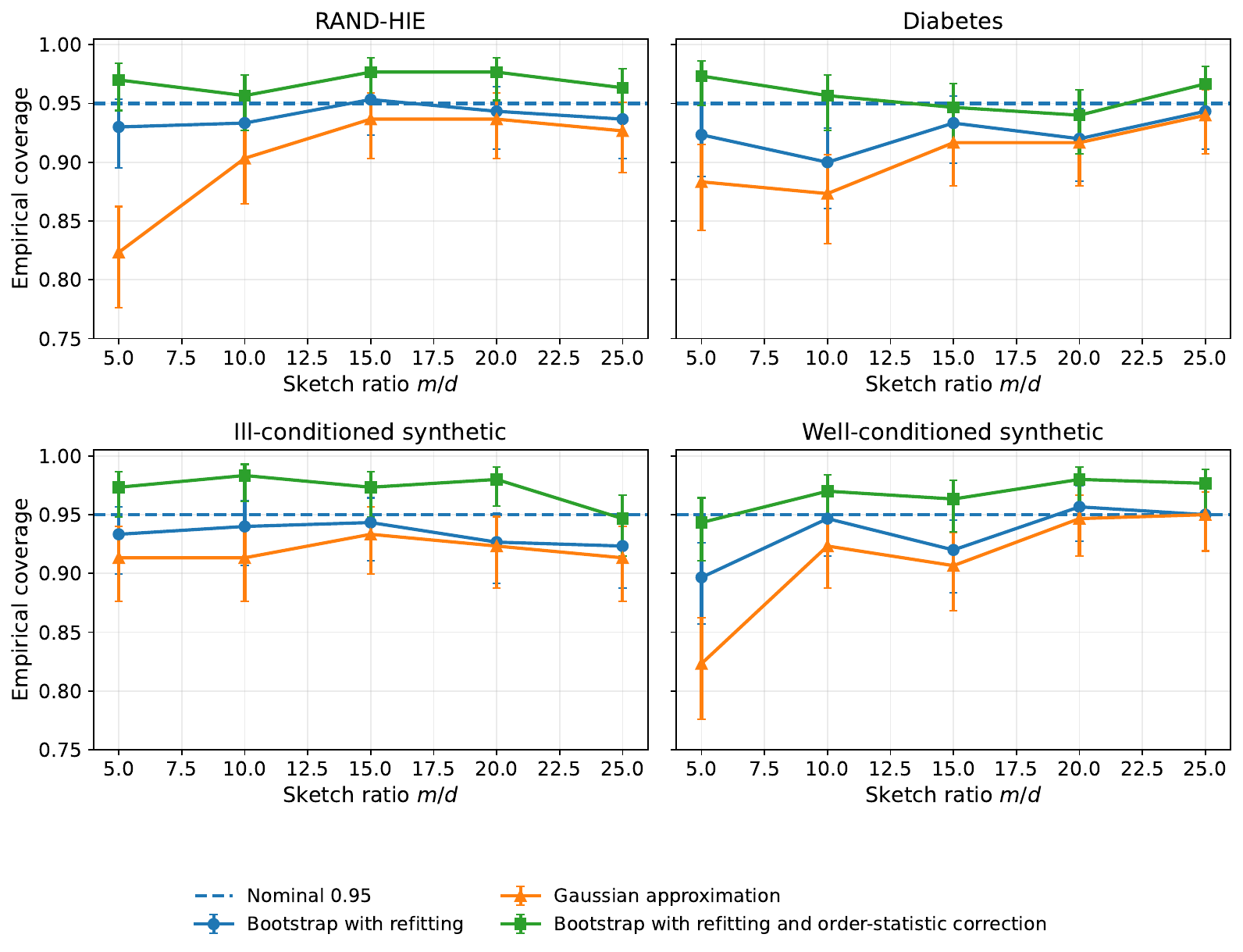}
\caption{Empirical coverage at nominal level $0.95$. The unadjusted and adjusted bootstrap bounds use exactly the same errors obtained by refitting; the adjustment changes only the selected order statistic. The Gaussian approximation uses the empirical influence covariance from the same observed sketch. Error bars are $95\%$ Wilson intervals.}
\label{fig:coverage}
\end{figure}

At $m/d=15$, the bootstrap with refitting has coverage
$0.920$--$0.953$, and its mean bound is within $0.8\%$ of the
independent-sketch reference quantile. The Gaussian approximation has coverage
$0.907$--$0.937$, with a mean bound $1.1\%$--$3.3\%$ below that quantile.
The order-statistic correction raises coverage to $0.947$--$0.977$, with a
mean bound $4.9\%$--$7.3\%$ above that quantile. The unadjusted
bound obtained by refitting is closest in average size to the independent-sketch
reference quantile. The adjusted bound is the most conservative.
\Cref{tab:coverage-summary} gives the data-set-specific values.

\begin{table}[t]
\centering
\caption{Coverage and relative size of the reported error bound at $m/d=15$, $B=199$, and nominal coverage $0.95$. The independent-sketch reference quantile is the $0.95$ quantile of coefficient errors estimated from independent sketches. ``Unadjusted'' uses the empirical quantile of errors obtained by refitting. ``Adjusted'' applies the order-statistic correction to the same errors.}
\label{tab:coverage-summary}
\resizebox{\linewidth}{!}{%
\begin{tabular}{lrrrrrr}
\toprule
& \multicolumn{3}{c}{Empirical coverage} & \multicolumn{3}{c}{Mean bound / independent-sketch quantile} \\
\cmidrule(lr){2-4}\cmidrule(lr){5-7}
Data set & Unadjusted & Gaussian & Adjusted & Unadjusted & Gaussian & Adjusted \\
\midrule
RAND-HIE & 0.953 & 0.937 & 0.977 & 0.999 & 0.967 & 1.073 \\
Diabetes & 0.933 & 0.917 & 0.947 & 0.992 & 0.973 & 1.064 \\
Ill-conditioned synthetic & 0.943 & 0.933 & 0.973 & 1.001 & 0.989 & 1.069 \\
Well-conditioned synthetic & 0.920 & 0.907 & 0.963 & 0.992 & 0.969 & 1.049 \\
\bottomrule
\end{tabular}}
\end{table}

\subsection{Accuracy and speed of the linearized bootstrap}
\label{sec:computational-results}

The second experiment isolates the method used to compute the bootstrap
replicate errors. We compare refitting and linearization using the
same multinomial weights and the same unadjusted empirical quantile.

The comparison uses a synthetic problem with $n=3000$, $d=40$,
$\operatorname{cond}(X^\top X)=10^4$, and $B=149$. In each repetition,
the relative difference between the two reported bounds is
\[
  \frac{
  \left|\widetilde\epsilon_{n,m,B}(\alpha)
  -\widehat\epsilon_{n,m,B}(\alpha)\right|
  }{
  \widehat\epsilon_{n,m,B}(\alpha)
  },
\]
where $\widehat\epsilon_{n,m,B}(\alpha)$ is the bound obtained by refitting.

The two error bounds are close. Across all tested sketch ratios, the median
relative difference is $1.45\%$, and the $90$th percentile is $4.01\%$. The
median ratio of replicate-generation time with refitting to that with
linearization is $2.11$. This ratio excludes the initial fit,
which is required by both methods. Results at each sketch ratio are reported
in \cref{tab:linearized-summary,fig:computation}. Details of the timing
measurements and computing environment are reported in
\appref{sec:supp-experiments}.

\begin{table}[!t]
\centering
\caption{Difference between the bootstrap bounds obtained by refitting and by linearization over $60$ repetitions at each sketch ratio. The runtime ratio is replicate-generation time with refitting divided by the corresponding time with linearization; the common initial fit is excluded.}
\label{tab:linearized-summary}
\begingroup\small\setlength{\tabcolsep}{7pt}
\begin{tabular}{rrrr}
\toprule
$m/d$ & Median relative difference & $90$th percentile & Median runtime ratio \\
\midrule
5  & 2.08\% & 5.26\% & 2.67$\times$ \\
10 & 1.41\% & 4.01\% & 2.11$\times$ \\
20 & 1.05\% & 2.73\% & 1.81$\times$ \\
\bottomrule
\end{tabular}
\endgroup
\end{table}

\begin{figure}[t]
\centering
\includegraphics[width=0.82\linewidth]{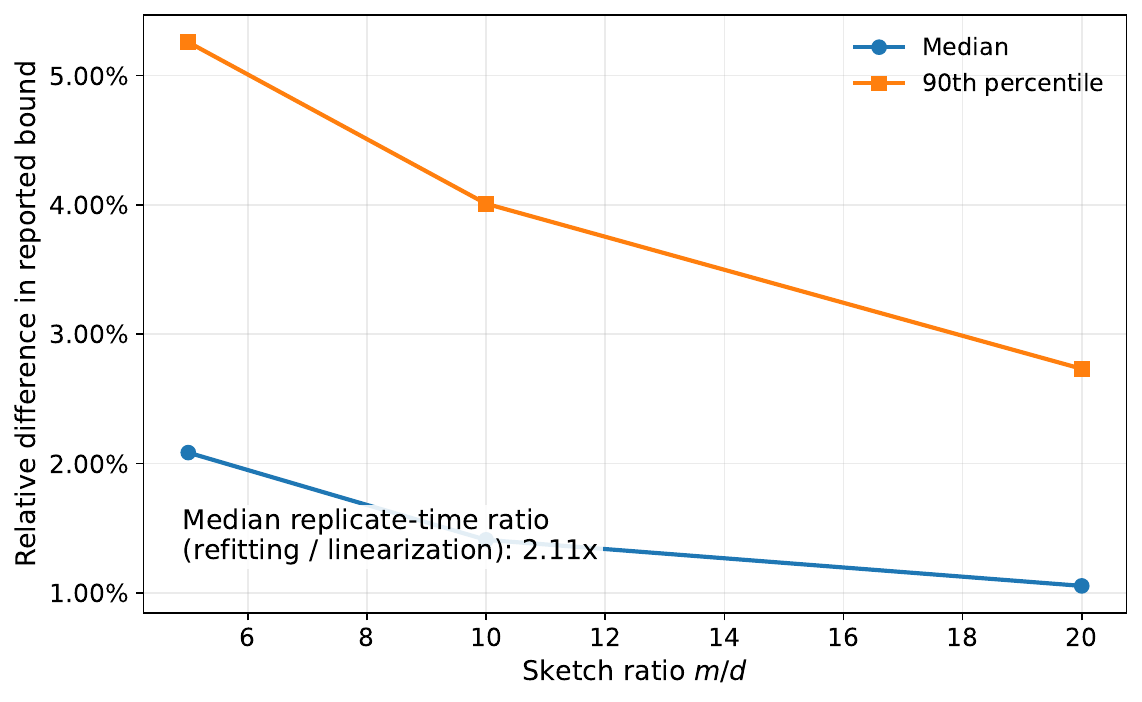}
\caption{Relative difference between the bootstrap bounds obtained by linearization and by refitting. Both methods use the same bootstrap weights and the same unadjusted empirical quantile, so the only difference is whether the ridge problem is refitted in each replicate. The annotation reports the median ratio of replicate-generation time with refitting to that with linearization.}
\label{fig:computation}
\end{figure}

\subsection{Effect of the order-statistic correction}
\label{sec:finite-b-experiment}

The third experiment isolates estimation of the bootstrap quantile. All
replicate errors in this experiment are computed by refitting. For each of
$20$ RAND-HIE sketches at $m/d=15$, we use $50{,}000$ bootstrap errors to
obtain a high-precision estimate of the conditional $0.95$ bootstrap
quantile. This quantity differs from the independent-sketch reference
quantile used in the coverage experiment. Holding the sketch fixed, we then
repeat the quantile calculation $300$ times using a smaller number $B$ of
bootstrap replicates. Each repetition generates one sequence of $499$
bootstrap errors. The result for each smaller value of $B$ uses the
corresponding prefix of that sequence.

We use $B\in\{99,199,499\}$. For $p=0.95$, these values make the rank
$k=\lceil Bp\rceil$ of the unadjusted empirical quantile satisfy
$k/(B+1)=0.95$; the ranks are $95$, $190$, and $475$, respectively.

For these three values of $B$, the unadjusted empirical quantile is below the
$B=50{,}000$ estimate in $0.435$--$0.471$ of the repetitions. After applying
the order-statistic correction, this probability decreases to $0.025$--$0.042$,
which is below the target $\delta=0.05$. Thus, in this experiment, the
correction makes underestimation of the conditional bootstrap quantile rare.

For $B=20$ and $49$, the corresponding probabilities for the ordinary
empirical quantile are $0.730$ and $0.553$. At these values, no order statistic
can satisfy the distribution-free correction rule because
$(1-\alpha)^B>\delta$.

\begin{figure}[t]
\centering
\includegraphics[width=0.82\linewidth]{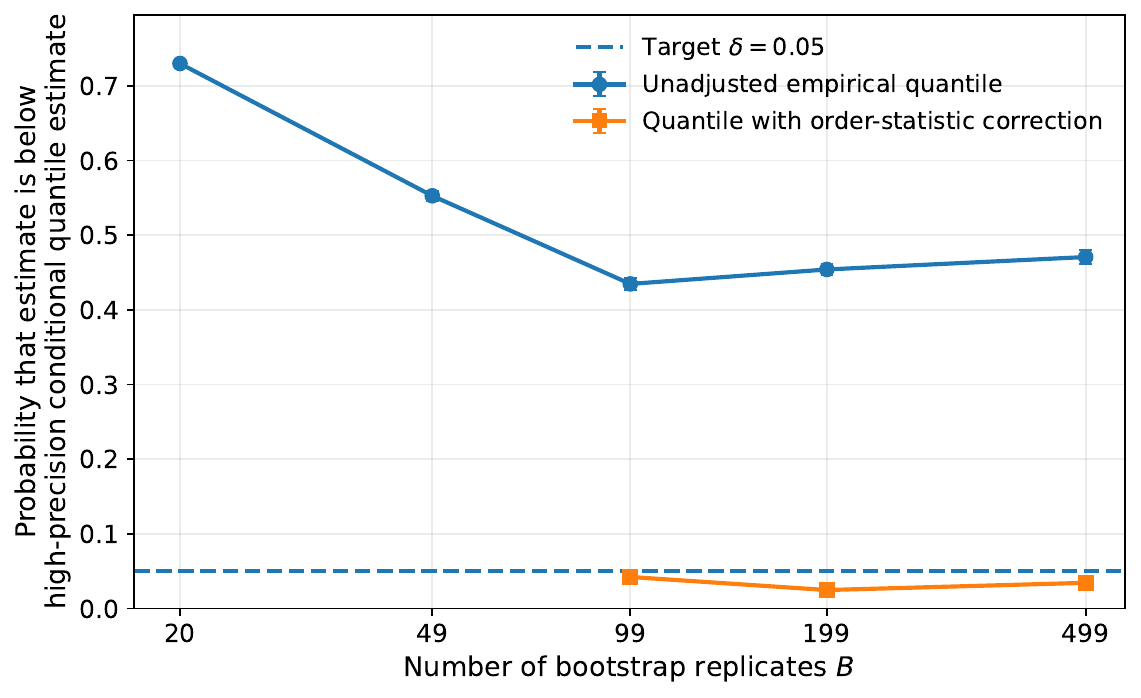}
\caption{Probability that an estimate based on $B$ bootstrap replicates is below the high-precision estimate of the conditional bootstrap quantile based on $50{,}000$ replicates. Both curves use errors obtained by refitting. One curve uses the unadjusted empirical quantile. The other uses the order-statistic correction.}
\label{fig:finite-b-correction}
\end{figure}

\subsection{Effect of a multiple-comparison adjustment on sketch-size selection}

The fourth question is how an adjustment for multiple comparisons across all
candidate sizes affects the selected sketch size. We compare three rules. The
\emph{Bonferroni-adjusted rule} uses nominal noncoverage probability $\alpha/K$
at each of the $K$ candidate sizes, as in \cref{eq:selection-rule}. The
\emph{unadjusted rule} uses nominal noncoverage probability $\alpha$ at every candidate, with no
adjustment for multiple comparisons. The \emph{pilot-based extrapolation} is
defined in \cref{eq:adaptive-m}. It uses the smallest candidate, $m_0=5d$,
as the pilot. The resulting estimate is rounded upward to the smallest
candidate size not below it. No candidate is returned if the estimate exceeds
$30d$.

On RAND-HIE, the candidate ratios are
$m/d\in\{5,10,15,20,25,30\}$ and $B=199$. We construct three tolerance values
by estimating the $0.95$ coefficient-error quantile from independent sketches
at $m/d=10$, $15$, and $20$. For example, when the tolerance is based on
$m/d=15$, it is set equal to the estimated $0.95$ error quantile at that
ratio. For each tolerance, a method selects the smallest candidate whose
reported error bound does not exceed the tolerance. In each repetition, we
generate the largest Gaussian sketch once and use its first $m$ compressed
rows at each candidate size. Bootstrap resampling is performed separately at
each size. Each tolerance
setting uses $300$ repetitions, and each tolerance is estimated from $3000$
independent sketches.

The Bonferroni-adjusted and unadjusted rules use the empirical
bootstrap error bounds defined in \cref{eq:exact-bootstrap-threshold}. The
order-statistic correction in \cref{eq:mc-threshold} is not applied. This
experiment therefore isolates the effect of adjusting for multiple
comparisons.

With $K=6$, $\alpha=0.05$, and $B=199$, no order statistic satisfying
\cref{eq:mc-order} exists at nominal noncoverage probability $\alpha/K$ for $\delta=0.05$, because
$(1-\alpha/K)^B>\delta$. Moreover, because $B=199$ is fixed, this
implementation is not guaranteed to satisfy the assumption in
\cref{thm:simultaneous-selection} that the error bound at each candidate size
has marginal coverage at least $1-\alpha/K$. At nominal noncoverage probability $\alpha/K$, the
generalized-inverse empirical quantile has rank
$\lceil199(1-0.05/6)\rceil=198$. If, hypothetically, the conditional
bootstrap-error distribution were continuous, this order statistic would
cover an additional independent draw with probability only $198/200=0.99$,
which is below $1-0.05/6=0.991\overline{6}$. The results below are therefore an empirical
comparison for a fixed number of bootstrap replicates. They do not establish
finite-sample familywise error control.

We record the fraction of all repetitions in which a method selects a
candidate whose actual coefficient error exceeds the tolerance. This
fraction is smallest for the Bonferroni-adjusted rule, ranging from $0.003$ to
$0.007$. This rule selects larger sketches and occasionally returns no
candidate for the strictest tolerance. The unadjusted rule generally selects
smaller sketches, but the corresponding fraction is higher, ranging from
$0.020$ to $0.040$. Pilot-based extrapolation lies between these two
procedures in this experiment. Thus, adjustment for multiple comparisons
leads to larger selected sketches. This experiment does not compare the total
computational cost of the three procedures.

\begin{figure}[!t]
\centering
\includegraphics[width=\linewidth]{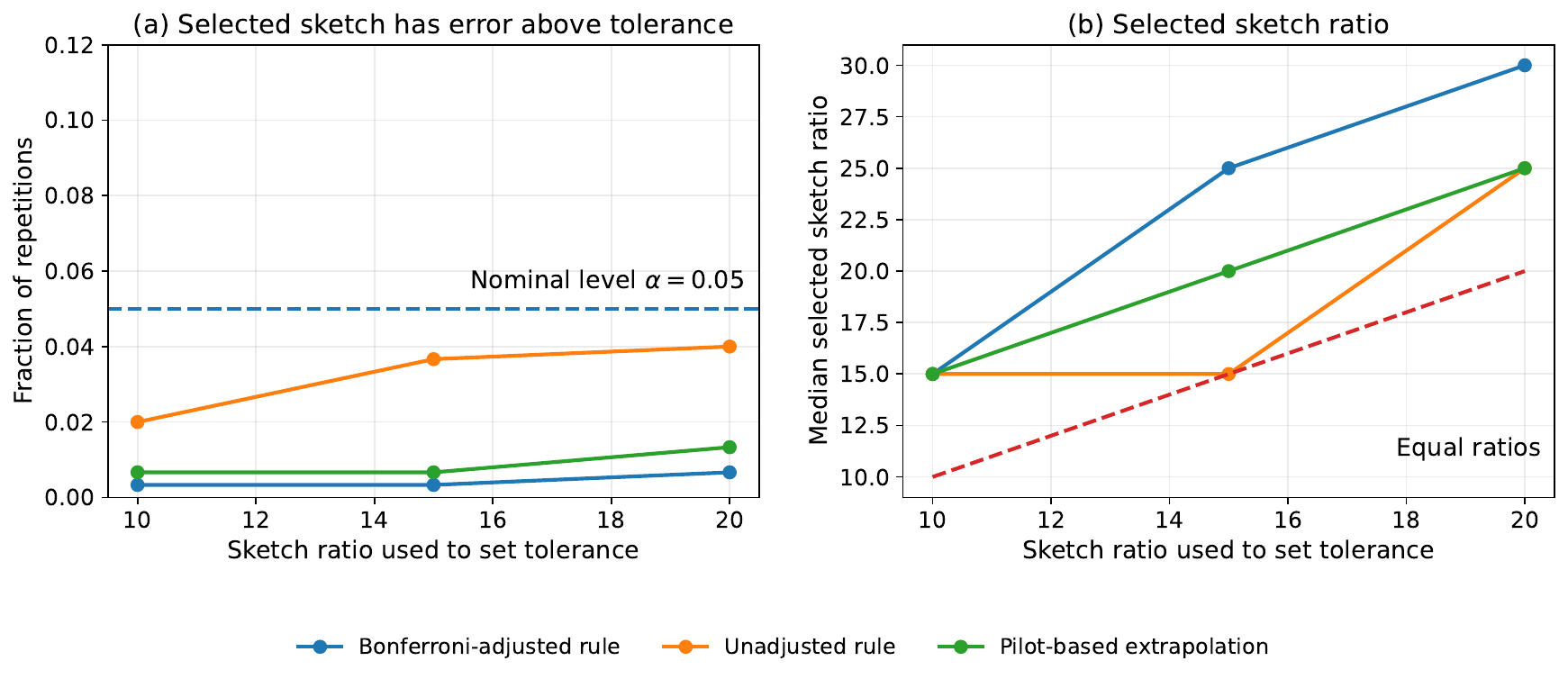}
\caption{Sketch-size selection on RAND-HIE. The horizontal axis gives the
sketch ratio at which the independent-sketch $0.95$ error quantile was
computed to set the tolerance. Left: fraction of all $300$ repetitions in
which a candidate is selected and its actual coefficient error exceeds the
tolerance. Repetitions with no selected candidate are included in the
denominator and not counted in the numerator. Right: median selected ratio
among repetitions in which a candidate was selected.}
\label{fig:certified-selection}
\end{figure}

\begin{table}[t]
\centering
\caption{Sketch-size selection over $300$ repetitions. The selection rate is
the fraction of repetitions in which a candidate is selected. The next column
is the fraction of all repetitions in which a candidate is selected and its
actual coefficient error exceeds the tolerance. The median selected ratio
is computed only from repetitions in which a candidate was selected.}
\label{tab:adaptive-summary}
\begingroup\small\setlength{\tabcolsep}{3pt}
\begin{tabularx}{\linewidth}{r>{\raggedright\arraybackslash}Xccc}
\toprule
\shortstack{Sketch ratio used\\to set tolerance} & Method & \shortstack{Selection\\rate} & \shortstack{Selected with error\\above tolerance} & \shortstack{Median selected\\ratio} \\
\midrule
10 & Bonferroni-adjusted rule & 1.000 & 0.003 & 15 \\
10 & Unadjusted rule & 1.000 & 0.020 & 15 \\
10 & Pilot-based extrapolation & 1.000 & 0.007 & 15 \\
\hline
15 & Bonferroni-adjusted rule & 1.000 & 0.003 & 25 \\
15 & Unadjusted rule & 1.000 & 0.037 & 15 \\
15 & Pilot-based extrapolation & 0.997 & 0.007 & 20 \\
\hline
20 & Bonferroni-adjusted rule & 0.847 & 0.007 & 30 \\
20 & Unadjusted rule & 1.000 & 0.040 & 25 \\
20 & Pilot-based extrapolation & 0.907 & 0.013 & 25 \\
\bottomrule
\end{tabularx}
\endgroup
\end{table}

\section{Discussion}
\label{sec:discussion}

The practical interpretation of the reported error bound depends on four
issues: the accuracy of the bootstrap approximation at a finite sketch size,
the accuracy and cost of linearization, Monte Carlo error in the estimated
quantile, and multiple comparisons across candidate sketch sizes.

\subsection{Interpretation of the bootstrap error bound}

The quantity of interest is $\norm{\bbetahat_{n,m}-\bbeta_n}_2$, the Euclidean
distance between the sketched and full-data ridge solutions. All probability
statements are conditional on the observed full data. They therefore account
for randomness in the sketch and bootstrap resampling, but not for sampling
uncertainty in the original data.

The bootstrap resamples only the compressed design--response rows
from the observed sketch. Under the assumptions of
\cref{thm:bootstrap-consistency,cor:bootstrap-distribution-consistency}, the
conditional distribution of the bootstrap coefficient error consistently
approximates the distribution of the coefficient error due to sketching.
\Cref{thm:exact-coverage} then gives asymptotically exact coverage when the
sketched estimator and bootstrap error bound are computed from the same
sketch. In \cref{sec:coverage-results}, the average bound is close to the
independent-sketch reference quantile at the tested ratios, although some undercoverage remains
at the smallest ratios. Thus, the bootstrap approximation can still have
nonnegligible error when $m$ is small.

\subsection{Linearized computation and the order-statistic correction}

Linearization avoids solving a new ridge problem for every bootstrap replicate
by approximating the coefficient change to first order.
\Cref{prop:linearized-equivalence} shows that, for any fixed bootstrap
replicate, the difference between the changes obtained by linearization and by refitting is
$o_p^*(m_n^{-1/2})$ in $\PS$-probability. In
\cref{sec:computational-results}, the two error bounds
are close, and generating the linearized replicates is faster. The measured
speedup depends on the implementation and problem dimensions.

After the replicate errors have been computed, the order-statistic correction
uses a higher order statistic to reduce the probability of underestimating
the exact conditional bootstrap quantile. For errors obtained by refitting,
\cref{prop:mc-protected} states that, conditional on the observed sketch, the
adjusted error bound is at least this quantile with probability at least
$1-\delta$ over the bootstrap draws, whenever the required order statistic
exists. This result controls Monte Carlo error caused by using finitely many
bootstrap replicates. It does not remove the asymptotic approximation error
between the bootstrap distribution and the distribution of the coefficient
error due to sketching. The experiment in \cref{sec:finite-b-experiment}
verifies the intended one-sided behavior and shows the resulting increase in
the error bound. When $(1-\alpha)^B>\delta$, no order statistic satisfies the
correction rule in \cref{eq:mc-order}.

\subsection{Sketch-size selection and sketch comparison}

Examining several candidate sketch sizes requires an adjustment for multiple
comparisons because the selected size depends on the error bounds computed at
all candidates. \Cref{thm:simultaneous-selection} applies a Bonferroni
adjustment to a fixed finite grid and asymptotically controls the probability
of selecting a candidate whose coefficient error exceeds the tolerance. In
the numerical comparison, the Bonferroni-adjusted rule gives the smallest such
probability, but selects larger sketches and occasionally returns no candidate
under the strictest tolerance. The theorem assumes marginal coverage of at
least $1-\alpha/K$ separately at every candidate size. For error bounds based
on unadjusted empirical bootstrap quantiles, this condition is obtained asymptotically by letting the number of
bootstrap replicates tend to infinity at each candidate size. Because the
experiment fixes $B$, it provides an empirical comparison rather than
finite-sample familywise error control. The unadjusted rule and
pilot-based extrapolation are included for comparison.

The covariance formula in \cref{eq:gamma} shows how the covariance of the
leading error term depends on the residuals, design matrix, regularization parameter, and
fourth moment of the sketch entries. \Cref{cor:rademacher-optimal} shows that
Rademacher entries minimize this covariance in the Loewner order among
sketches with i.i.d. standardized entries. This asymptotic result does
not cover uniform row sampling or structured sketches. Those methods and
finite-sample comparisons require separate analysis. The modest differences in \cref{fig:sketch-comparison} do
not contradict the ordering, because an asymptotic comparison of leading terms
need not produce large finite-sample differences.

\subsection{Scope and possible extensions}

The proofs use a fixed-dimensional asymptotic regime and i.i.d. isotropic sketch rows. Proportional high-dimensional regimes can exhibit different bootstrap behavior \citep{clarte2024bootstrap}. Structured transforms such as SRHT and CountSketch induce dependence among compressed rows and would require resampling schemes that preserve that dependence.

The quantity studied here is Euclidean coefficient error relative to the full-data ridge solution. Prediction error, objective-value error, and selected linear functionals would require separate influence representations and coverage arguments. Population-sampling uncertainty and floating-point error add further layers beyond the fixed-data, exact-arithmetic analysis. Sequential or anytime-valid methods may also reduce the conservativeness of Bonferroni selection over a finite grid.

\section{Conclusion}
\label{sec:conclusion}

Building on the bootstrap of compressed design--response pairs for randomized least squares, this paper
establishes the corresponding validity theory for sketched ridge regression
with a varying normalized regularization parameter. Under i.i.d. isotropic sketch rows,
the theory gives an asymptotic linear representation, Gaussian limits for the
sketched estimator and conditional bootstrap distribution, uniform
consistency of the bootstrap distribution, and asymptotically exact
coverage when the estimator and error bound use the same sketch. An explicit
covariance formula also shows that Rademacher entries minimize the covariance
of the leading error term in the Loewner order among sketches with i.i.d.
standardized entries.

For practical computation, the linearized bootstrap approximates bootstrap
refitting without solving a new ridge problem in every replicate. The
order-statistic correction controls underestimation of the exact conditional
bootstrap quantile. The Bonferroni adjustment accounts for selection over a
fixed grid of candidate sketch sizes. In the numerical experiments, the mean
error bound obtained by refitting is close to the coefficient-error quantile
estimated from independent sketches. The bounds obtained by linearization and
refitting are also close, while linearized replicate generation is faster.
The Bonferroni-adjusted rule selects a candidate whose actual error exceeds
the tolerance less often than the unadjusted rule, at the cost of
larger selected sketches.

\section*{Data and materials availability}

The LaTeX source accompanies this submission. No separate code or data archive
is provided. The numerical experiments use two data sets distributed with
Python packages and two generated synthetic data sets; no manual data download
is required. \appref{sec:source-experiments} documents historical benchmark
settings whose original scripts and outputs are unavailable. Those unavailable
artifacts are not presented as reproducible outputs.

\section*{Declarations}

\textbf{Conflict of interest.} The authors declare no conflict of interest.

\bibliographystyle{plainnat}
\bibliography{references}


\appendix
\numberwithin{equation}{section}
\numberwithin{figure}{section}
\numberwithin{table}{section}
\numberwithin{algorithm}{section}

\section{Notation used in the proofs}
\label{app:proofs}

Throughout this appendix, all asymptotic limits are taken with
\[
  n\to\infty,
  \qquad
  d\ \text{fixed},
  \qquad
  m=m_n\to\infty.
\]
The additional condition $B=B_n\to\infty$ is used only for results involving
the empirical bootstrap distribution or its unadjusted empirical quantile.
For convenience, we restate the notation used repeatedly in the proofs. Each
citation points to the equation where the quantity is first defined in the
main text.

\paragraph{Basic operations and probability.}
For $\bm x=(x_1,\ldots,x_k)^\top\in\R^k$,
$M=(M_{j\ell})\in\R^{r\times k}$, an event $\mathcal A$, and
$\bm a,\bm b\in\R^k$,
\[
  \norm{\bm x}_2:=\bigg(\sum_{j=1}^kx_j^2\bigg)^{1/2},
  \qquad
  \norm{M}_{\mathrm F}:=\bigg(\sum_{j=1}^r\sum_{\ell=1}^kM_{j\ell}^2\bigg)^{1/2},
  \qquad
  \norm{M}_{\op}:=\sup_{\substack{\bm x\in\R^k\\\norm{\bm x}_2=1}}\norm{M\bm x}_2,
\]
\[
  \mathbf1\{\mathcal A\}:=
  \begin{cases}
    1,&\mathcal A\text{ occurs},\\
    0,&\mathcal A\text{ does not occur},
  \end{cases}
  \qquad
  I_k:=\bigl(\mathbf1\{j=\ell\}\bigr)_{j,\ell=1}^k,
  \qquad
  \bm e_j:=\bigl(\mathbf1\{\ell=j\}\bigr)_{\ell=1}^k,
\]
\[
  \diag(\bm a):=\bigl(a_j\mathbf1\{j=\ell\}\bigr)_{j,\ell=1}^k,
  \qquad
  \bm a\odot\bm b:=\bigl(a_jb_j\bigr)_{j=1}^k.
\]
For $N=(N_{j\ell})\in\R^{k\times k}$,
\[
  \tr(N):=\sum_{j=1}^kN_{jj},
  \qquad
  \rank(N):=\dim\{N\bm x:\bm x\in\R^k\}.
\]
For symmetric $N,L\in\R^{k\times k}$,
\[
  \lambda_{\min}(N):=\min_{\substack{\bm x\in\R^k\\\norm{\bm x}_2=1}}\bm x^\top N\bm x,
  \qquad
  N\succeq L
  \quad\Longleftrightarrow\quad
  \bm x^\top(N-L)\bm x\geq0
  \quad\text{for every }\bm x\in\R^k.
\]
Let $\Pp$ be the underlying probability measure.  For a square-integrable
random vector $\bm Y$,
\[
  \E\bm Y:=\int\bm Y\,d\Pp,
  \qquad
  \operatorname{Var}(\bm Y)
  :=\E\bigl[(\bm Y-\E\bm Y)(\bm Y-\E\bm Y)^\top\bigr].
\]

\paragraph{Full-data quantities.}
The definitions in \cref{eq:full-ridge,eq:full-moments} are
\[
  X_n\in\R^{n\times d},
  \qquad
  \by_n\in\R^n,
  \qquad
  \lambda_n>0,
\]
\[
  \bbeta_n
  :=\argmin_{\bbeta\in\R^d}
  \left\{
    \frac1n\norm{X_n\bbeta-\by_n}_2^2
    +\lambda_n\norm{\bbeta}_2^2
  \right\},
\]
\[
  H_n:=\frac1nX_n^\top X_n,
  \qquad
  \bg_n:=\frac1nX_n^\top\by_n,
  \qquad
  A_n:=(H_n+\lambda_nI_d)^{-1},
  \qquad
  \bbeta_n=A_n\bg_n.
\]
The residual and the full-data normal equation used throughout the
proofs are
\[
  \br_n:=\by_n-X_n\bbeta_n,
  \qquad
  (H_n+\lambda_nI_d)\bbeta_n=\bg_n,
  \qquad
  \frac1nX_n^\top\br_n=\lambda_n\bbeta_n.
\]
The limits and uniform invertibility condition in \cref{ass:regime} are
\[
  \lambda_n\to\lambda_\infty\geq0,
  \qquad
  H_n\to H,
  \qquad
  \bg_n\to\bg,
  \qquad
  A:=(H+\lambda_\infty I_d)^{-1},
  \qquad
  \bbeta:=A\bg,
\]
\[
  \exists\,c_0>0,\ \exists\,n_0\in\mathbb N:
  \quad
  \lambda_{\min}(H_n+\lambda_nI_d)\geq c_0
  \quad\text{for all }n\geq n_0.
\]

\paragraph{Sketch probability space and compressed quantities.}
For each $n$ and $m$,
\[
  \bs_{n,1},\ldots,\bs_{n,m}\in\R^n\ \text{are i.i.d.},
  \qquad
  \mathcal S_{n,m}:=\sigma(\bs_{n,1},\ldots,\bs_{n,m}),
  \qquad
  \mathcal S_n:=\mathcal S_{n,m_n}. 
\]
Here, $\sigma(\bs_{n,1},\ldots,\bs_{n,m})$ is the smallest sigma-algebra with respect to
which all sketch rows $\bs_{n,1},\ldots,\bs_{n,m}$ are measurable.
For $\mathcal A\in\mathcal S_{n,m}$ and every integrable
$\mathcal S_{n,m}$-measurable random variable $Z$,
\[
  \PS(\mathcal A):=\Pp(\mathcal A),
  \qquad
  \ES Z:=\E Z,
  \qquad
  \ES(\bs_{n,i}\bs_{n,i}^\top)=I_n.
\]
The definitions in
\cref{eq:sketch-matrix,eq:compressed-rows,eq:sketched-moments,eq:sketched-ridge,eq:sketched-residuals}
are
\[
  S_{n,m}:=\frac1{\sqrt m}
  \begin{bmatrix}
    \bs_{n,1}^\top\\[-1mm]
    \vdots\\[-1mm]
    \bs_{n,m}^\top
  \end{bmatrix}\in\R^{m\times n},
  \qquad
  \bz_{n,i}:=\frac1{\sqrt n}X_n^\top\bs_{n,i}\in\R^d,
  \qquad
  u_{n,i}:=\frac1{\sqrt n}\by_n^\top\bs_{n,i}\in\R,
\]
\[
  \widehat H_{n,m}:=\frac1m\sum_{i=1}^m\bz_{n,i}\bz_{n,i}^\top\in\R^{d\times d},
  \qquad
  \widehat\bg_{n,m}:=\frac1m\sum_{i=1}^m\bz_{n,i}u_{n,i}\in\R^d,
\]
\[
  \widehat A_{n,m}:=(\widehat H_{n,m}+\lambda_nI_d)^{-1}\in\R^{d\times d},
  \qquad
  \bbetahat_{n,m}:=\widehat A_{n,m}\widehat\bg_{n,m}\in\R^d,
\]
\[
  \widehat r_{n,m,i}:=u_{n,i}-\bz_{n,i}^\top\bbetahat_{n,m}\in\R.
\]
The sketched normal equation is
\[
  (\widehat H_{n,m}+\lambda_nI_d)\bbetahat_{n,m}
  =\widehat\bg_{n,m}.
\]

\paragraph{Conditional probability for bootstrap resampling.}
For every event $\mathcal A$ and every integrable random variable $Z$ defined
on the joint probability space for the sketch and bootstrap resampling,
\[
  \Pstar(\mathcal A):=\Pp(\mathcal A\mid\mathcal S_{n,m}),
  \qquad
  \Estar Z:=\E(Z\mid\mathcal S_{n,m}).
\]
In asymptotic statements, the conditioning sigma-algebra is $\mathcal S_n$.

\paragraph{Bootstrap quantities computed by refitting.}
For $b=1,\ldots,B$, the definitions in
\cref{eq:multinomial-weights,eq:weighted-moments,eq:exact-bootstrap-estimator,eq:exact-bootstrap-threshold}
are
\[
  \bw_1,\ldots,\bw_B\mid\mathcal S_{n,m}
  \overset{\mathrm{i.i.d.}}{\sim}
  \operatorname{Multinomial}
  \left(m;\frac1m,\ldots,\frac1m\right),
  \qquad
  \bw_b:=(w_{b,1},\ldots,w_{b,m})^\top,
\]
\[
  \widehat H_{n,m,b}^*
  :=\frac1m\sum_{i=1}^m w_{b,i}\bz_{n,i}\bz_{n,i}^\top,
  \qquad
  \widehat\bg_{n,m,b}^*
  :=\frac1m\sum_{i=1}^m w_{b,i}\bz_{n,i}u_{n,i},
\]
\[
  \bbetastar_{n,m,b}
  :=(\widehat H_{n,m,b}^*+\lambda_nI_d)^{-1}
    \widehat\bg_{n,m,b}^*,
  \qquad
  e_{n,m,b}^*:=\norm{\bbetastar_{n,m,b}-\bbetahat_{n,m}}_2,
\]
\[
  \widehat\epsilon_{n,m,B}(\alpha)
  :=\inf\left\{t\geq0:
  \frac1B\sum_{b=1}^B\mathbf1\{e_{n,m,b}^*\leq t\}
  \geq1-\alpha\right\}.
\]
For a distribution function $F$ and $p\in(0,1)$, the quantile convention is
\[
  F^{-1}(p):=\inf\{t\in\R:F(t)\geq p\}.
\]

\paragraph{Linearized bootstrap quantities.}
The definitions in
\cref{eq:linearized-delta,eq:empirical-influence,eq:linearized-as-influence,eq:linearized-error}
are
\[
  \widetilde\Delta_{n,m,b}^*
  :=\widehat A_{n,m}
  \left\{
    (\widehat\bg_{n,m,b}^*-\widehat\bg_{n,m})
    -(\widehat H_{n,m,b}^*-\widehat H_{n,m})\bbetahat_{n,m}
  \right\},
\]
\[
  \widehat{\bm\psi}_{n,m,i}
  :=\widehat A_{n,m}
  \left\{\bz_{n,i}\widehat r_{n,m,i}
  -\lambda_n\bbetahat_{n,m}\right\},
  \qquad i=1,\ldots,m,
\]
\[
  \frac1m\sum_{i=1}^m\widehat{\bm\psi}_{n,m,i}=\bzero,
  \qquad
  \widetilde\Delta_{n,m,b}^*
  =\frac1m\sum_{i=1}^m(w_{b,i}-1)\widehat{\bm\psi}_{n,m,i},
\]
\[
  \widetilde e_{n,m,b}^*:=\norm{\widetilde\Delta_{n,m,b}^*}_2,
  \qquad
  \widetilde\epsilon_{n,m,B}(\alpha)
  :=\inf\left\{t\geq0:
  \frac1B\sum_{b=1}^B\mathbf1\{\widetilde e_{n,m,b}^*\leq t\}
  \geq1-\alpha\right\}.
\]

\paragraph{Influence vectors and their covariance.}
The theoretical influence vector defined in
\cref{eq:derivative-map,eq:single-influence} is
\begin{align*}
  \bm\psi_{n,i}
  &:=A_n\left\{\bz_{n,i}
  (u_{n,i}-\bz_{n,i}^\top\bbeta_n)
  -\lambda_n\bbeta_n\right\}
  =\frac1nA_nX_n^\top
  (\bs_{n,i}\bs_{n,i}^\top-I_n)\br_n.
\end{align*}  
It satisfies
\begin{align*}
  \ES\bm\psi_{n,i}&=\bzero.
\end{align*}  
Its covariance, defined in \cref{eq:omega}, is
\begin{align*}
  \Omega_n
  &:=\ES\left[
  (\bm\psi_{n,1}-\ES\bm\psi_{n,1})
  (\bm\psi_{n,1}-\ES\bm\psi_{n,1})^\top
  \right]
  =\ES(\bm\psi_{n,1}\bm\psi_{n,1}^\top),
  \qquad
  \Omega_n\to\Omega.
\end{align*}
For sketches with i.i.d. entries in \cref{prop:covariance},
\[
  \E s=0,
  \qquad
  \E s^2=1,
  \qquad
  \E|s|^4<\infty,
  \qquad
  \kappa_s:=\E s^4-3.
\]

\paragraph{Conditional resampling representation.}
For the first bootstrap replicate at $m=m_n$, define
\[
  I_{n,1}^*,\ldots,I_{n,m_n}^*
  \mid\mathcal S_n
  \overset{\mathrm{i.i.d.}}{\sim}
  \operatorname{Uniform}\{1,\ldots,m_n\},
  \qquad
  w_{1,i}:=\sum_{j=1}^{m_n}\mathbf1\{I_{n,j}^*=i\},
  \quad i=1,\ldots,m_n.
\]
Then
\[
  (w_{1,1},\ldots,w_{1,m_n})^\top
  \mid\mathcal S_n
  \sim\operatorname{Multinomial}
  \left(m_n;\frac1{m_n},\ldots,\frac1{m_n}\right).
\]

\paragraph{Distribution functions and limiting quantile.}
The definitions in
\cref{eq:ideal-distribution-functions,eq:empirical-bootstrap-cdf} are
\[
  F_{n,m_n}(t)
  :=\PS\!\left\{
  \sqrt{m_n}\norm{\bbetahat_{n,m_n}-\bbeta_n}_2\leq t
  \right\},
\]
\[
  \widehat F_{n,m_n}(t)
  :=\Pstar\!\left\{
  \sqrt{m_n}\norm{\bbetastar_{n,m_n,1}-\bbetahat_{n,m_n}}_2\leq t
  \right\},
\]
\[
  \widehat F_{n,m_n,B_n}(t)
  :=\frac1{B_n}\sum_{b=1}^{B_n}
  \mathbf1\!\left\{\sqrt{m_n}\,e_{n,m_n,b}^*\leq t\right\}.
\]
The Gaussian limit and its $(1-\alpha)$-quantile used in
\appref{app:distribution-proof} and \appref{app:coverage-proof} are
\[
  \bm Z\sim N(\bzero,\Omega),
  \qquad
  G(t):=\Pp\{\norm{\bm Z}_2\leq t\},
  \qquad
  q:=G^{-1}(1-\alpha)
  =\inf\{t\in\R:G(t)\geq1-\alpha\}.
\]

\paragraph{Conditional bootstrap quantile and one-sided order-statistic correction.}
For $p:=1-\alpha$, the conditional $p$-quantile defined in
\cref{sec:mc-protection} is
\[
  q_{n,m}^*(p)
  :=\inf\left\{t\geq0:
  \Pstar\!\left(
  \norm{\bbetastar_{n,m,1}-\bbetahat_{n,m}}_2\leq t
  \right)\geq p\right\}.
\]
This is the exact conditional quantile of the bootstrap error before it is
estimated from finitely many bootstrap replicates. The rank used for the
one-sided correction and the resulting adjusted error bound are defined in
\cref{eq:mc-order,eq:mc-threshold}:
\[
  k_B(p,\delta)
  :=\min\left\{k\in\{1,\ldots,B\}:
  \Pp\{\operatorname{Bin}(B,p)\leq k-1\}\geq1-\delta
  \right\},
\]
with $k_B(p,\delta):=\infty$ when the displayed set is empty, and
\[
  e_{n,m,(1)}^*\leq\cdots\leq e_{n,m,(B)}^*,
  \qquad
  \widehat\epsilon_{n,m,B}^{\mathrm{MC}}(\alpha,\delta)
  :=
  \begin{cases}
    e_{n,m,(k_B(1-\alpha,\delta))}^*,
      &k_B(1-\alpha,\delta)<\infty,\\
    +\infty,&k_B(1-\alpha,\delta)=\infty.
  \end{cases}
\]

\paragraph{Selection over a finite grid.}
For each $n$, let $m_{n,1}<\cdots<m_{n,K}$ denote the ordered grid in
\cref{thm:simultaneous-selection}. The index $j$ identifies a sketch size,
and $\widehat j_n=0$ means that no sketch size is selected. For
$j=1,\ldots,K$,
\[
  m_{n,j}\to\infty,
  \qquad
  \bbetahat_{n,j}:=\bbetahat_{n,m_{n,j}},
\]
\[
  \liminf_{n\to\infty}
  \Pp\left\{
  \norm{\bbetahat_{n,j}-\bbeta_n}_2
  \leq\widehat\epsilon_{n,j}(\alpha/K)
  \right\}
  \geq1-\frac\alpha K,
\]
\[
  \widehat j_n
  :=
  \begin{cases}
    \displaystyle
    \min\left\{j\in\{1,\ldots,K\}:
    \widehat\epsilon_{n,j}(\alpha/K)\leq\tau\right\},
    &\text{if the set is nonempty},\\[2mm]
    0,&\text{if the set is empty}.
  \end{cases}
\]

\paragraph{Stochastic orders and convergence.}
Let $a_n>0$ be deterministic.  For random vectors or matrices $\bm Y_n$ and
bootstrap random vectors or matrices $\bm Y_n^*$, with $\norm{\cdot}$
denoting the relevant norm,
\[
  \bm Y_n=O_p(a_n)
  \quad\Longleftrightarrow\quad
  \forall\eta>0\ \exists M<\infty:
  \limsup_{n\to\infty}
  \Pp\{\norm{\bm Y_n}>Ma_n\}\leq\eta,
\]
\[
  \bm Y_n=o_p(a_n)
  \quad\Longleftrightarrow\quad
  \forall\varepsilon>0:
  \Pp\{\norm{\bm Y_n}>\varepsilon a_n\}\to0,
\]
\[
  \bm Y_n^*=O_p^*(a_n)\ \text{in }\PS\text{-probability}
  \quad\Longleftrightarrow\quad
  \forall\eta>0\ \exists M<\infty:
  \PS\!\left\{
  \Pstar(\norm{\bm Y_n^*}>Ma_n)>\eta
  \right\}\to0,
\]
\[
  \bm Y_n^*=o_p^*(a_n)\ \text{in }\PS\text{-probability}
  \quad\Longleftrightarrow\quad
  \forall\varepsilon>0\ \forall\eta>0:
  \PS\!\left\{
  \Pstar(\norm{\bm Y_n^*}>\varepsilon a_n)>\eta
  \right\}\to0.
\]
For random vectors $\bm Y_n,\bm Y$,
\[
  \bm Y_n\pto\bm Y
  \quad\Longleftrightarrow\quad
  \forall\varepsilon>0:
  \Pp\{\norm{\bm Y_n-\bm Y}>\varepsilon\}\to0,
\]
\[
  \bm Y_n\dto\bm Y
  \quad\Longleftrightarrow\quad
  \E h(\bm Y_n)\to\E h(\bm Y)
  \quad\text{for every bounded continuous }h.
\]
Conditional weak convergence in $\PS$-probability is denoted by
\[
  \bm Y_n^*\dto\bm Y
  \quad\text{conditionally in }\PS\text{-probability},
\]
and is defined by
\[
  \Estar h(\bm Y_n^*)\pto\E h(\bm Y)
  \quad\text{for every bounded Lipschitz }h.
\]
\section{\texorpdfstring{Proof of \Cref{lem:concrete-conditions}: sufficient conditions for sketch-row regularity}{Proof of sufficient conditions for sketch-row regularity}}
\label{app:sufficient-conditions}

\begin{proof}[Proof of \Cref{lem:concrete-conditions}]
Recall the full-data quantities in \cref{eq:full-ridge,eq:full-moments},
the compressed rows in \cref{eq:compressed-rows}, and the notation in
\appref{app:proofs}. Let
\[
  \bs:=(s_1,\ldots,s_n)^\top\overset{d}{=}\bs_{n,1},
  \qquad
  C=C^\top\in\R^{d\times d},
  \qquad
  \bm c\in\R^d.
\]
Here, $C$ and $\bm c$ are arbitrary. The corresponding scalar linear
combination of the centered compressed moments is
\begin{equation}
  \frac1n\left[
  \bs^\top X_nCX_n^\top\bs-\tr(X_nCX_n^\top)
  +(\bs^\top X_n\bm c)(\bs^\top\by_n)-\bm c^\top X_n^\top\by_n
  \right].
  \label{eq:scalar-quadratic-form}
\end{equation}
Because a quadratic form depends only on the symmetric part of its matrix,
the expression in \cref{eq:scalar-quadratic-form} is the centered quadratic
form associated with the symmetric matrix
\begin{align*}
  X_n C X_n^\top
  +
  \frac{
    X_n\bm c\,\by_n^\top
    +
    \by_n\bm c^\top X_n^\top
  }{2}.
\end{align*}
For any symmetric matrix $K\in\mathbb R^{n\times n}$ and any random vector
$\bs$ with independent and identically distributed entries satisfying
\[
  \E[s_1]=0,
  \qquad
  \E[s_1^2]=1,
  \qquad
  \E[s_1^4]<\infty,
\]
the standard variance identity for quadratic forms is
\begin{align*}
  \mathrm{Var}\!\left[
    \bs^\top K\bs-\tr K
  \right]
  =
  2\tr(K^2)
  +
  \bigl\{\E[s_1^4]-3\bigr\}
  \sum_{j=1}^n K_{jj}^2.
\end{align*}
Applying this identity with
\[
  K
  =
  X_n C X_n^\top
  +
  \frac{
    X_n\bm c\,\by_n^\top
    +
    \by_n\bm c^\top X_n^\top
  }{2}
\]
and accounting for the factor $1/n$ in
\cref{eq:scalar-quadratic-form}, we obtain
\begin{align}
  &\phantom{=}\mathrm{Var}\!\left[
    \frac{1}{n}
    \left\{
      \bs^\top K\bs-\tr K
    \right\}
  \right] \nonumber \\
  &=
  \frac{2}{n^2}
  \tr\left(
    \left[
      X_n C X_n^\top
      +
      \frac{
        X_n\bm c\,\by_n^\top
        +
        \by_n\bm c^\top X_n^\top
      }{2}
    \right]^2
  \right) \nonumber\\
  &\quad+
  \frac{\E[s_1^4]-3}{n^2}
  \sum_{j=1}^n
  \left[
    \bm e_j^\top
    \left\{
      X_n C X_n^\top
      +
      X_n\bm c\,\by_n^\top
    \right\}
    \bm e_j
  \right]^2.
  \label{eq:scalar-quadratic-variance}
\end{align}
In the second term, the symmetrization can be omitted because the two cross
terms have the same diagonal entries:
\begin{align*}
  \bm e_j^\top
  \frac{
    X_n\bm c\,\by_n^\top
    +
    \by_n\bm c^\top X_n^\top
  }{2}
  \bm e_j
  =
  \bm e_j^\top
  X_n\bm c\,\by_n^\top
  \bm e_j.
\end{align*}
Thus, \cref{eq:diagonal-condition} implies convergence of the second term in
\cref{eq:scalar-quadratic-variance}.
Expanding the trace in the first term gives
\begin{align*}
  &\phantom{=}\frac1{n^2}\tr\left[
     X_nCX_n^\top+
     \frac{X_n\bm c\,\by_n^\top+\by_n\bm c^\top X_n^\top}{2}
     \right]^2\\
  &=\tr(CH_nCH_n)+2\bg_n^\top CH_n\bm c
    +\frac12\left\{(\bm c^\top\bg_n)^2
    +\frac{\by_n^\top\by_n}{n}\,\bm c^\top H_n\bm c\right\},
\end{align*}
which converges under the stated assumptions. Hence the variance of every
scalar projection converges. Since the dimension is fixed, the polarization
identity implies convergence of the covariance matrix. Finally,
$\bm\psi_{n,1}=A_n\{(\bz_{n,1}u_{n,1}-\bg_n)
-(\bz_{n,1}\bz_{n,1}^\top-H_n)\bbeta_n\}$. Combining this covariance
convergence with $A_n\to A$ and $\bbeta_n\to\bbeta$ from
\cref{ass:regime} proves that $\Omega_n$ converges as required in
\cref{eq:omega}.

By \citet[Lemma~B.26]{bai2010spectral} and the eighth-moment assumption,
there is a constant $C_0<\infty$ such that the fourth moment of
\cref{eq:scalar-quadratic-form} satisfies
\begin{equation}
\begin{aligned}
  \E\left|
    \frac1n\{\bs^\top K\bs-\tr K\}
  \right|^4
  &\leq \frac{C_0}{n^4}
  \left[\{\tr(K^2)\}^2+\tr(K^4)\right]\\
  &\leq \frac{2C_0}{n^4}\{\tr(K^2)\}^2
  =O(1),
\end{aligned}
  \label{eq:quadratic-fourth-bound}
\end{equation}
where $\tr(K^4)\leq\{\tr(K^2)\}^2$ because $K$ is symmetric, and the last
step follows from the preceding trace expansion. Applying
this bound to the finitely many matrix and vector coordinates yields
\cref{eq:design-regularity}.  Hence every part of \cref{ass:sketch} holds.
\end{proof}

\section{\texorpdfstring{Proof of \Cref{prop:covariance}: linearization and covariance}{Proof of linearization and covariance}}
\label{app:covariance-proof}

\begin{proof}[Proof of \Cref{prop:covariance}]
Recall the full-data and sketched moments in
\cref{eq:full-moments,eq:moment-vector}, the estimators in
\cref{eq:full-ridge,eq:sketched-ridge}, and the influence vector
$\bm\psi_{n,i}$ in \cref{eq:derivative-map,eq:single-influence}. We first
establish the asymptotic linear representation. By
\cref{eq:moment-vector}, each moment error is an average of independent,
centered random matrices or vectors of fixed dimension. The uniform
fourth-moment bound in \cref{ass:sketch} therefore implies
\begin{align*}
  \ES\norm{\widehat H_{n,m_n}-H_n}_{\mathrm F}^{2}
  &=\frac{1}{m_n}\ES
    \norm{\bz_{n,1}\bz_{n,1}^{\top}-H_n}_{\mathrm F}^{2}
    =O(m_n^{-1}),\\
  \ES\norm{\widehat\bg_{n,m_n}-\bg_n}_{2}^{2}
  &=\frac{1}{m_n}\ES
    \norm{\bz_{n,1}u_{n,1}-\bg_n}_{2}^{2}
    =O(m_n^{-1}).
\end{align*}
Consequently,
\begin{equation}
  \norm{\widehat H_{n,m_n}-H_n}_{\op}=O_p(m_n^{-1/2}),
  \qquad
  \norm{\widehat\bg_{n,m_n}-\bg_n}_2=O_p(m_n^{-1/2}),
  \label{eq:target-moment-rates}
\end{equation}
where we used
$\norm{M}_{\op}\leq\norm{M}_{\mathrm F}$.

The full-data and sketched normal equations are
\[
  (H_n+\lambda_nI_d)\bbeta_n=\bg_n,
  \qquad
  (\widehat H_{n,m_n}+\lambda_nI_d)
  \bbetahat_{n,m_n}=\widehat\bg_{n,m_n}.
\]
Subtracting the first equation from the second and collecting the terms
that multiply $\bbetahat_{n,m_n}-\bbeta_n$ gives
\begin{equation}
\begin{aligned}
  \bbetahat_{n,m_n}-\bbeta_n
  ={}&(\widehat H_{n,m_n}+\lambda_nI_d)^{-1}
  \bigl[(\widehat\bg_{n,m_n}-\bg_n)
        -(\widehat H_{n,m_n}-H_n)\bbeta_n\bigr].
\end{aligned}
  \label{eq:target-exact-difference}
\end{equation}
The resolvent identity applied with reference matrix
$H_n+\lambda_nI_d$ gives
\begin{equation}
\begin{aligned}
  (\widehat H_{n,m_n}+\lambda_nI_d)^{-1}
  =A_n-A_n(\widehat H_{n,m_n}-H_n)
  (\widehat H_{n,m_n}+\lambda_nI_d)^{-1}.
\end{aligned}
  \label{eq:target-resolvent}
\end{equation}
Substitution of \cref{eq:target-resolvent} into
\cref{eq:target-exact-difference} yields
\begin{align}
  \bbetahat_{n,m_n}-\bbeta_n
  &=A_n\bigl[(\widehat\bg_{n,m_n}-\bg_n)
       -(\widehat H_{n,m_n}-H_n)\bbeta_n\bigr]
       +R_{n,m_n},
  \label{eq:beta-expansion}
\end{align}
where 
\begin{align*}
    R_{n,m_n}
  &:=-A_n(\widehat H_{n,m_n}-H_n)
  (\widehat H_{n,m_n}+\lambda_nI_d)^{-1}
  \bigl[(\widehat\bg_{n,m_n}-\bg_n)
       -(\widehat H_{n,m_n}-H_n)\bbeta_n\bigr].
\end{align*}
By \cref{ass:regime},
$\norm{A_n}_{\op}\leq c_0^{-1}$ for all sufficiently large $n$.
Moreover, \cref{eq:target-moment-rates} implies
$\norm{\widehat H_{n,m_n}-H_n}_{\op}=o_p(1)$.  Hence, with probability
tending to one,
$\norm{\widehat H_{n,m_n}-H_n}_{\op}\leq c_0/2$; on this event,
Weyl's inequality gives
\[
  \lambda_{\min}
  (\widehat H_{n,m_n}+\lambda_nI_d)
  \geq c_0/2,
\]
so that
$\norm{(\widehat H_{n,m_n}+\lambda_nI_d)^{-1}}_{\op}
\leq2/c_0$.
Because $\bbeta_n\to\bbeta$ under \cref{ass:regime}, the sequence
$\{\norm{\bbeta_n}_2\}$ is bounded.  On the same event,
\begin{align}
  \norm{R_{n,m_n}}_2
  &\leq
  \norm{A_n}_{\op}
  \norm{\widehat H_{n,m_n}-H_n}_{\op}
  \norm{(\widehat H_{n,m_n}+\lambda_nI_d)^{-1}}_{\op}
  \notag\\
  &\quad\times
  \left\{
    \norm{\widehat\bg_{n,m_n}-\bg_n}_2
    +\norm{\widehat H_{n,m_n}-H_n}_{\op}
     \norm{\bbeta_n}_2
  \right\}\notag\\
  &\leq\frac{2}{c_0^2}
  \norm{\widehat H_{n,m_n}-H_n}_{\op}
  \left\{
    \norm{\widehat\bg_{n,m_n}-\bg_n}_2
    +\norm{\widehat H_{n,m_n}-H_n}_{\op}
     \norm{\bbeta_n}_2
  \right\}\notag\\
  &=O_p(m_n^{-1}).
  \label{eq:target-remainder}
\end{align}

We now identify the first-order term in \cref{eq:beta-expansion}.
Using \cref{eq:moment-vector},
\begin{align}
  &(\widehat\bg_{n,m_n}-\bg_n)
       -(\widehat H_{n,m_n}-H_n)\bbeta_n\notag\\
  &\quad=\frac1{m_n}\sum_{i=1}^{m_n}
  \left[
    \bz_{n,i}u_{n,i}-\bg_n
    -(\bz_{n,i}\bz_{n,i}^{\top}-H_n)\bbeta_n
  \right]\notag\\
  &\quad=\frac1{m_n}\sum_{i=1}^{m_n}
  \left[
    \bz_{n,i}(u_{n,i}-\bz_{n,i}^{\top}\bbeta_n)
    -(\bg_n-H_n\bbeta_n)
  \right].
  \label{eq:leading-moment-rearrangement}
\end{align}
The full-data normal equation implies
$\bg_n-H_n\bbeta_n=\lambda_n\bbeta_n$.  Hence, by the definition of
$\bm\psi_{n,i}$ in \cref{eq:derivative-map},
\begin{equation}
  A_n\bigl[(\widehat\bg_{n,m_n}-\bg_n)
       -(\widehat H_{n,m_n}-H_n)\bbeta_n\bigr]
  =\frac1{m_n}\sum_{i=1}^{m_n}\bm\psi_{n,i}.
  \label{eq:leading-as-influences}
\end{equation}
Combining
\cref{eq:beta-expansion,eq:target-remainder,eq:leading-as-influences}
gives
\[
  \sqrt{m_n}(\bbetahat_{n,m_n}-\bbeta_n)
  =\frac1{\sqrt{m_n}}\sum_{i=1}^{m_n}\bm\psi_{n,i}
   +\sqrt{m_n}R_{n,m_n}.
\]
Since
$\sqrt{m_n}R_{n,m_n}=O_p(m_n^{-1/2})=o_p(1)$,
this proves \cref{eq:target-linearization}.

It remains to compute the covariance when the sketch entries are i.i.d.
Using the scalar entry variable $s$ in \cref{eq:sketch-moments},
define
\[
  \bs:=(s_1,\ldots,s_n)^\top,
  \qquad
  s_1,\ldots,s_n\overset{\mathrm{i.i.d.}}{\sim}s,
  \qquad
  \bm\psi_n(\bs)
  :=\frac1nA_nX_n^\top(\bs\bs^\top-I_n)\br_n.
\]
By \cref{eq:omega},
$\Omega_n=\E\{\bm\psi_n(\bs)\bm\psi_n(\bs)^\top\}$.
Since $\E(\bs\bs^{\top})=I_n$,
$\E\{\bm\psi_n(\bs)\}=\bzero$, and therefore
\begin{equation}
\begin{aligned}
  \Omega_n
  =\frac1{n^2}A_nX_n^{\top}
  \E\left[
    (\bs\bs^{\top}-I_n)\br_n\br_n^{\top}
    (\bs\bs^{\top}-I_n)
  \right]
  X_nA_n^{\top}.
\end{aligned}
  \label{eq:influence-covariance-reduction}
\end{equation}
Thus, we only need to evaluate the matrix expectation in
\cref{eq:influence-covariance-reduction}.  Define
\[
  r_{n,j}:=\bm e_j^\top\br_n,
  \qquad
  \bm v:=(\bs\bs^{\top}-I_n)\br_n,
  \qquad
  a:=\bs^{\top}\br_n.
\]
By the definition of $\bm v$, \cref{eq:single-influence} becomes
$\bm\psi_n(\bs)=n^{-1}A_nX_n^{\top}\bm v$.
The $j$th component of $\bm v$ is
\[
  v_j=s_ja-r_{n,j}.
\]
Because the coordinates of $\bs$ are independent, centered, and have unit
variance,
\begin{equation}
  \E(s_ja)
  =\sum_{\ell=1}^{n}r_{n,\ell}\E(s_js_\ell)
  =r_{n,j}.
  \label{eq:sja-mean}
\end{equation}
For $j\neq k$,
\begin{align*}
  \E(v_jv_k)
  &=\E\{(s_ja-r_{n,j})(s_ka-r_{n,k})\}\\
  &=\E(s_js_ka^2)-r_{n,j}r_{n,k},
\end{align*}
where \cref{eq:sja-mean} was used in the second equality.  Expanding
\[
  a^2
  =\sum_{\ell=1}^{n}\sum_{h=1}^{n}
    r_{n,\ell}r_{n,h}s_\ell s_h,
\]
we obtain
\[
  \E(s_js_ka^2)
  =\sum_{\ell=1}^{n}\sum_{h=1}^{n}
    r_{n,\ell}r_{n,h}\E(s_js_ks_\ell s_h).
\]
When $j\neq k$, independence and centering imply that a summand can be
nonzero only if every index among $j,k,\ell,h$ occurs at least twice.
The only possibilities are $(\ell,h)=(j,k)$ and $(\ell,h)=(k,j)$.
Since $\E(s_j^2s_k^2)=1$, it follows that
\[
  \E(s_js_ka^2)=2r_{n,j}r_{n,k},
\]
\begin{equation}
  \E(v_jv_k)=r_{n,j}r_{n,k},
  \qquad j\neq k.
  \label{eq:quadratic-moment-offdiagonal}
\end{equation}

For a diagonal entry, write
\[
  a=r_{n,j}s_j+\sum_{\ell\neq j}r_{n,\ell}s_\ell.
\]
The random variable $s_j$ is independent of the sum on the right, and the
sum has mean zero.  Therefore
\begin{align*}
  \E(s_j^2a^2)
  &=r_{n,j}^2\E(s_j^4)
    +\E(s_j^2)\E\bigg(
      \sum_{\ell\neq j}r_{n,\ell}s_\ell
    \bigg)^2\\
  &=r_{n,j}^2\E(s^4)
    +\sum_{\ell\neq j}r_{n,\ell}^2\\
  &=\norm{\br_n}_2^2
    +\{\E(s^4)-1\}r_{n,j}^2.
\end{align*}
Using \cref{eq:sja-mean} once more,
\begin{align}
  \E(v_j^2)
  &=\E\{(s_ja-r_{n,j})^2\}\notag\\
  &=\E(s_j^2a^2)-2r_{n,j}\E(s_ja)+r_{n,j}^2\notag\\
  &=\norm{\br_n}_2^2
    +\{\E(s^4)-2\}r_{n,j}^2.
  \label{eq:quadratic-moment-diagonal}
\end{align}
Recall from \cref{prop:covariance} that $\kappa_s:=\E(s^4)-3$. Combining the
off-diagonal result in \cref{eq:quadratic-moment-offdiagonal} with the diagonal
result in \cref{eq:quadratic-moment-diagonal} gives
\begin{equation}
\begin{aligned}
  &\phantom{=}\E\left[
    (\bs\bs^{\top}-I_n)\br_n\br_n^{\top}
    (\bs\bs^{\top}-I_n)
  \right]
  &=\norm{\br_n}_2^2I_n
    +\br_n\br_n^{\top}
    +\kappa_s\diag(\br_n\odot\br_n),
\end{aligned}
  \label{eq:quadratic-moment}
\end{equation}

Substituting \cref{eq:quadratic-moment} into
\cref{eq:influence-covariance-reduction} gives
\begin{align*}
  \Omega_n
  =A_n\Biggl\{
    &\frac{\norm{\br_n}_2^2}{n}H_n
    +\left(\frac{X_n^{\top}\br_n}{n}\right)
     \left(\frac{X_n^{\top}\br_n}{n}\right)^{\top}
    +\frac{\kappa_s}{n^2}
      X_n^{\top}\diag(\br_n\odot\br_n)X_n
  \Biggr\}A_n^{\top}.
\end{align*}
Finally, the full-data normal equation yields
\[
  \frac1nX_n^{\top}\br_n
  =\bg_n-H_n\bbeta_n
  =\lambda_n\bbeta_n.
\]
Therefore the middle matrix product in the displayed expression for
$\Omega_n$ equals $\lambda_n^2\bbeta_n\bbeta_n^{\top}$.  Substitution
produces exactly the right-hand side of \cref{eq:gamma} and completes
the proof.
\end{proof}

\section{\texorpdfstring{Proof of \Cref{prop:linearized-equivalence}: first-order equivalence}{Proof of first-order equivalence}}
\label{app:linearized-proof}

\begin{proof}[Proof of \Cref{prop:linearized-equivalence}]
Recall the weighted moments and bootstrap estimator in
\cref{eq:weighted-moments,eq:exact-bootstrap-estimator}, the inverse
$\widehat A_{n,m_n}$ in \cref{eq:sketched-residuals}, and the linearized
change $\widetilde\Delta_{n,m_n,1}^{*}$ in
\cref{eq:linearized-delta}. We show that this linearized change is the
first-order term in the bootstrap estimator obtained by refitting. It is
enough to consider the first bootstrap replicate, so every starred quantity
below has replicate index $b=1$. We expand around the observed sketched
moments $(\widehat H_{n,m_n},\widehat\bg_{n,m_n})$. The derivative at this
point is given in \cref{eq:realized-derivative}.

The bootstrap and sketched normal equations are
\[
  (\widehat H_{n,m_n,1}^{*}+\lambda_nI_d)\bbetastar_{n,m_n,1}
  =\widehat\bg_{n,m_n,1}^{*},
  \qquad
  (\widehat H_{n,m_n}+\lambda_nI_d)\bbetahat_{n,m_n}
  =\widehat\bg_{n,m_n}.
\]
Subtracting the second equation from the first gives
\[
  (\widehat H_{n,m_n,1}^{*}+\lambda_nI_d)
  (\bbetastar_{n,m_n,1}-\bbetahat_{n,m_n})
  =(\widehat\bg_{n,m_n,1}^{*}-\widehat\bg_{n,m_n})
  -(\widehat H_{n,m_n,1}^{*}-\widehat H_{n,m_n})\bbetahat_{n,m_n}.
\]
Consequently,
\begin{align}
  \bbetastar_{n,m_n,1}-\bbetahat_{n,m_n}
  &=(\widehat H_{n,m_n,1}^{*}+\lambda_nI_d)^{-1}
  \left\{(\widehat\bg_{n,m_n,1}^{*}-\widehat\bg_{n,m_n})
  -(\widehat H_{n,m_n,1}^{*}-\widehat H_{n,m_n})\bbetahat_{n,m_n}\right\}.
  \label{eq:bootstrap-exact-difference}
\end{align}
The resolvent identity centered at
$\widehat H_{n,m_n}+\lambda_nI_d$ is
\begin{align}
  (\widehat H_{n,m_n,1}^{*}+\lambda_nI_d)^{-1}
  &=\widehat A_{n,m_n}
  -\widehat A_{n,m_n}(\widehat H_{n,m_n,1}^{*}-\widehat H_{n,m_n})
  (\widehat H_{n,m_n,1}^{*}+\lambda_nI_d)^{-1}.
  \label{eq:bootstrap-resolvent}
\end{align}
Substitution into \cref{eq:bootstrap-exact-difference} shows that the
first-order term is
\begin{align*}
  \widehat A_{n,m_n}\left\{(\widehat\bg_{n,m_n,1}^{*}-\widehat\bg_{n,m_n})
  -(\widehat H_{n,m_n,1}^{*}-\widehat H_{n,m_n})\bbetahat_{n,m_n}\right\}
  =\widetilde\Delta_{n,m_n,1}^{*},
\end{align*}
where the equality is the definition in \cref{eq:linearized-delta}
with $m=m_n$ and $b=1$.  Define the remaining term by
\[
\begin{aligned}
  R_{n,m_n,1}^{*}
  :={}&-\widehat A_{n,m_n}
  (\widehat H_{n,m_n,1}^{*}-\widehat H_{n,m_n})
  (\widehat H_{n,m_n,1}^{*}+\lambda_nI_d)^{-1}\\
  &\quad\times\left\{(\widehat\bg_{n,m_n,1}^{*}-\widehat\bg_{n,m_n})
  -(\widehat H_{n,m_n,1}^{*}-\widehat H_{n,m_n})
  \bbetahat_{n,m_n}\right\}.
\end{aligned}
\]
Then \cref{eq:bootstrap-exact-difference,eq:bootstrap-resolvent} give the
exact expansion
\[
  \bbetastar_{n,m_n,1}-\bbetahat_{n,m_n}
  =\widetilde\Delta_{n,m_n,1}^{*}+R_{n,m_n,1}^{*}.
\]
Moreover,
\begin{align}
  \norm{R_{n,m_n,1}^{*}}_2
  \leq{}&
  \norm{\widehat A_{n,m_n}}_{\op}
  \norm{\widehat H_{n,m_n,1}^{*}-\widehat H_{n,m_n}}_{\op}
  \norm{(\widehat H_{n,m_n,1}^{*}+\lambda_nI_d)^{-1}}_{\op}\notag\\
  &\times\left\{
  \norm{\widehat\bg_{n,m_n,1}^{*}-\widehat\bg_{n,m_n}}_2
  +\norm{\widehat H_{n,m_n,1}^{*}-\widehat H_{n,m_n}}_{\op}
   \norm{\bbetahat_{n,m_n}}_2\right\}.
  \label{eq:bootstrap-remainder-bound}
\end{align}
To establish the conditional order of the remainder, represent the first
bootstrap sample as $m_n$ independent draws from the empirical distribution
of the observed compressed rows. Conditional on $\mathcal S_n$,
\begin{align*}
  \Estar\norm{\widehat H_{n,m_n,1}^{*}-\widehat H_{n,m_n}}_{\mathrm F}^2
  &=\frac1{m_n}\left\{\frac1{m_n}\sum_{i=1}^{m_n}
    \norm{\bz_{n,i}\bz_{n,i}^\top-\widehat H_{n,m_n}}_{\mathrm F}^2\right\},\\
  \Estar\norm{\widehat\bg_{n,m_n,1}^{*}-\widehat\bg_{n,m_n}}_2^2
  &=\frac1{m_n}\left\{\frac1{m_n}\sum_{i=1}^{m_n}
    \norm{\bz_{n,i}u_{n,i}-\widehat\bg_{n,m_n}}_2^2\right\}.
\end{align*}
The quantities in braces are $O_p(1)$ by \cref{ass:sketch}.
Markov's inequality and $\norm{M}_{\op}\leq\norm{M}_{\mathrm F}$ therefore give
\begin{equation}
  \norm{\widehat H_{n,m_n,1}^{*}-\widehat H_{n,m_n}}_{\op}
  =O_p^*(m_n^{-1/2}),
  \qquad
  \norm{\widehat\bg_{n,m_n,1}^{*}-\widehat\bg_{n,m_n}}_2
  =O_p^*(m_n^{-1/2})
  \label{eq:bootstrap-conditional-rates}
\end{equation}
in $\PS$-probability.  By \cref{ass:regime,prop:covariance},
$\bbetahat_{n,m_n}=O_p(1)$.  By
\cref{ass:regime,eq:target-moment-rates},
$\lambda_{\min}(\widehat H_{n,m_n}+\lambda_nI_d)$ is bounded away from
zero with probability tending to one.  On the event
$\lambda_{\min}(\widehat H_{n,m_n}+\lambda_nI_d)\geq c_0/2$,
Weyl's inequality gives
$\lambda_{\min}(\widehat H_{n,m_n,1}^{*}+\lambda_nI_d)\geq c_0/4$
whenever
$\norm{\widehat H_{n,m_n,1}^{*}-\widehat H_{n,m_n}}_{\op}\leq c_0/4$.
The first relation in \cref{eq:bootstrap-conditional-rates} implies that the
latter event has conditional probability tending to one in
$\PS$-probability. Hence
$(\widehat H_{n,m_n,1}^{*}+\lambda_nI_d)^{-1}=O_p^*(1)$ in
$\PS$-probability. Substituting these bounds into
\cref{eq:bootstrap-remainder-bound} gives
$R_{n,m_n,1}^{*}=O_p^*(m_n^{-1})$ in $\PS$-probability. Multiplying by
$\sqrt{m_n}$ proves \cref{eq:linearized-equivalence}.
\end{proof}

\section{\texorpdfstring{Proof of \Cref{thm:bootstrap-consistency}: Gaussian limits}{Proof of Gaussian limits}}
\label{app:bootstrap-proof}

\begin{proof}[Proof of \Cref{thm:bootstrap-consistency}]
Recall $A_n$ and $\bbeta_n$ from \cref{eq:full-ridge,eq:full-moments},
the influence vectors $\bm\psi_{n,i}$ and
$\widehat{\bm\psi}_{n,m_n,i}$ from
\cref{eq:derivative-map,eq:single-influence,eq:empirical-influence}, and the
covariance matrices from \cref{eq:omega}. We first prove asymptotic normality
of the sketched estimator and then prove the conditional bootstrap
approximation.

\paragraph{Central limit theorem for the sketched estimator.}
The matrices $A_n$ and vectors $\bbeta_n$ are bounded.  Therefore
\cref{ass:sketch,eq:derivative-map} imply
$\sup_n\ES\norm{\bm\psi_{n,1}}_2^4<\infty$.  For any
$\bm c\in\R^d$ and $\varepsilon>0$,
\begin{equation}
  \ES\left[(\bm c^\top\bm\psi_{n,1})^2
  \mathbf1\{|\bm c^\top\bm\psi_{n,1}|>\varepsilon\sqrt{m_n}\}\right]
  \leq\frac{\ES(\bm c^\top\bm\psi_{n,1})^4}{\varepsilon^2m_n}
  \longrightarrow0.
  \label{eq:lindeberg-bound}
\end{equation}
Thus the Lindeberg condition holds for every fixed $\bm c$. Since the
variance converges to $\bm c^\top\Omega\bm c$, the Lindeberg--Feller theorem
followed by the Cram\'er--Wold device gives
\begin{equation}
  \frac1{\sqrt{m_n}}\sum_{i=1}^{m_n}\bm\psi_{n,i}
  \dto N(\bzero,\Omega).
  \label{eq:moment-clt}
\end{equation}
Combining \cref{eq:moment-clt,eq:target-linearization} proves
\cref{eq:target-clt}.

\paragraph{Conditional central limit theorem for the bootstrap.}
Recall the estimated influence vectors
$\widehat{\bm\psi}_{n,m_n,i}$ from \cref{eq:empirical-influence}.  The
sketched normal equation implies
\[
  \frac1{m_n}\sum_{i=1}^{m_n}\widehat{\bm\psi}_{n,m_n,i}
  =\widehat A_{n,m_n}
  \{\widehat\bg_{n,m_n}-\widehat H_{n,m_n}\bbetahat_{n,m_n}
  -\lambda_n\bbetahat_{n,m_n}\}
  =\bzero,
\]
so their empirical mean is zero.

Recall the conditional resampling representation from
\appref{app:proofs}:
\[
  I_{n,1}^*,\ldots,I_{n,m_n}^*\mid\mathcal S_n
  \overset{\mathrm{i.i.d.}}{\sim}
  \operatorname{Uniform}\{1,\ldots,m_n\},
  \qquad
  w_{1,i}:=\sum_{j=1}^{m_n}\mathbf1\{I_{n,j}^*=i\},
  \quad i=1,\ldots,m_n.
\]
Then $(w_{1,1},\ldots,w_{1,m_n})$ has the law of the first multinomial
weight vector $\bw_1$ in \cref{eq:multinomial-weights}.  Hence,
conditional on $\mathcal S_n$,
\begin{align}
  \sqrt{m_n}\,\widetilde\Delta_{n,m_n,1}^{*}
  &=\frac1{\sqrt{m_n}}\sum_{i=1}^{m_n}(w_{1,i}-1)
    \widehat{\bm\psi}_{n,m_n,i}\notag\\
  &=\frac1{\sqrt{m_n}}\sum_{j=1}^{m_n}
    \widehat{\bm\psi}_{n,m_n,I_{n,j}^*}.
  \label{eq:bootstrap-iid-representation}
\end{align}
The first equality is \cref{eq:linearized-as-influence} with $b=1$.
For the second equality, the count representation gives
$\sum_iw_{1,i}\widehat{\bm\psi}_{n,m_n,i}
=\sum_j\widehat{\bm\psi}_{n,m_n,I_{n,j}^*}$, and the zero empirical
mean removes $\sum_i\widehat{\bm\psi}_{n,m_n,i}$.

We next verify the conditional covariance and Lindeberg conditions.
From \cref{eq:derivative-map,eq:empirical-influence},
\begin{align}
  \widehat{\bm\psi}_{n,m_n,i}-\bm\psi_{n,i}
  &=(\widehat A_{n,m_n}-A_n)
    \left\{\bz_{n,i}(u_{n,i}-\bz_{n,i}^\top\bbeta_n)
      -\lambda_n\bbeta_n\right\}\notag\\
  &\quad-\widehat A_{n,m_n}
    (\bz_{n,i}\bz_{n,i}^\top+\lambda_nI_d)
    (\bbetahat_{n,m_n}-\bbeta_n).
  \label{eq:empirical-influence-difference}
\end{align}
The centered representation used below follows directly from the
sketched normal equation.  Indeed, \cref{eq:empirical-influence} and
\cref{eq:sketched-residuals} give
\begin{align}
  \widehat{\bm\psi}_{n,m_n,i}
  &=\widehat A_{n,m_n}
  \{\bz_{n,i}u_{n,i}
  -(\bz_{n,i}\bz_{n,i}^\top+\lambda_nI_d)\bbetahat_{n,m_n}\}\notag\\
  &=\widehat A_{n,m_n}\left\{
  (\bz_{n,i}u_{n,i}-\widehat\bg_{n,m_n})
  -(\bz_{n,i}\bz_{n,i}^\top-\widehat H_{n,m_n})\bbetahat_{n,m_n}
  \right\},
  \label{eq:empirical-influence-centered}
\end{align}
because
$\widehat\bg_{n,m_n}=(\widehat H_{n,m_n}+\lambda_nI_d)
\bbetahat_{n,m_n}$.

The resolvent identity in \cref{eq:target-resolvent} and the rate in
\cref{eq:target-moment-rates} yield
$\widehat A_{n,m_n}-A_n=O_p(m_n^{-1/2})$, while
\cref{prop:covariance} gives
$\bbetahat_{n,m_n}-\bbeta_n=O_p(m_n^{-1/2})$.  From
\cref{eq:empirical-influence-difference},
\begin{align*}
  &\frac1{m_n}\sum_{i=1}^{m_n}
  \norm{\widehat{\bm\psi}_{n,m_n,i}-\bm\psi_{n,i}}_2^2\\
  &\quad\leq
  2\norm{\widehat A_{n,m_n}-A_n}_{\op}^2
  \frac1{m_n}\sum_{i=1}^{m_n}
  \norm{\bz_{n,i}(u_{n,i}-\bz_{n,i}^\top\bbeta_n)
  -\lambda_n\bbeta_n}_2^2\\
  &\qquad+2\norm{\widehat A_{n,m_n}}_{\op}^2
  \norm{\bbetahat_{n,m_n}-\bbeta_n}_2^2
  \frac1{m_n}\sum_{i=1}^{m_n}
  \norm{\bz_{n,i}\bz_{n,i}^\top+\lambda_nI_d}_{\op}^2.
\end{align*}
Both empirical averages on the right are $O_p(1)$ by
\cref{ass:regime,ass:sketch}, and
$\norm{\widehat A_{n,m_n}}_{\op}=O_p(1)$.  Therefore
\[
  \frac1{m_n}\sum_{i=1}^{m_n}
  \norm{\widehat{\bm\psi}_{n,m_n,i}-\bm\psi_{n,i}}_2^2\pto0.
\]
Also, the uniform fourth-moment bound gives
\[
  \norm{\frac1{m_n}\sum_{i=1}^{m_n}
  \bm\psi_{n,i}\bm\psi_{n,i}^\top-\Omega_n}_{\mathrm F}\pto0.
\]
Moreover, by Cauchy--Schwarz,
\begin{align*}
  &\norm{\frac1{m_n}\sum_{i=1}^{m_n}
  \{\widehat{\bm\psi}_{n,m_n,i}\widehat{\bm\psi}_{n,m_n,i}^\top
  -\bm\psi_{n,i}\bm\psi_{n,i}^\top\}}_{\mathrm F}\\
  &\leq
  \left\{\frac1{m_n}\sum_{i=1}^{m_n}
  \norm{\widehat{\bm\psi}_{n,m_n,i}-\bm\psi_{n,i}}_2^2\right\}^{1/2}\\
  &\quad\times\left[
  \left\{\frac1{m_n}\sum_{i=1}^{m_n}
  \norm{\widehat{\bm\psi}_{n,m_n,i}}_2^2\right\}^{1/2}
  +\left\{\frac1{m_n}\sum_{i=1}^{m_n}
  \norm{\bm\psi_{n,i}}_2^2\right\}^{1/2}
  \right]\pto0.
\end{align*}
Together with \cref{eq:omega}, this proves
\begin{equation}
  \frac1{m_n}\sum_{i=1}^{m_n}
  \widehat{\bm\psi}_{n,m_n,i}\widehat{\bm\psi}_{n,m_n,i}^\top
  \pto\Omega.
  \label{eq:empirical-covariance}
\end{equation}
By \cref{eq:empirical-influence-centered}, each influence vector is
written in terms of centered row moments.  This representation, the bounds
$\norm{\widehat A_{n,m_n}}_{\op}=O_p(1)$ and
$\norm{\bbetahat_{n,m_n}}_2=O_p(1)$, and the fourth-moment condition in
\cref{ass:sketch} imply
\begin{align*}
  &\phantom{=}\frac1{m_n}\sum_{i=1}^{m_n}
  \norm{\widehat{\bm\psi}_{n,m_n,i}}_2^4\\
  &\leq 8\norm{\widehat A_{n,m_n}}_{\op}^4
  \left\{
  \frac1{m_n}\sum_{i=1}^{m_n}
  \norm{\bz_{n,i}u_{n,i}-\widehat\bg_{n,m_n}}_2^4
  \quad+\norm{\bbetahat_{n,m_n}}_2^4
  \frac1{m_n}\sum_{i=1}^{m_n}
  \norm{\bz_{n,i}\bz_{n,i}^\top-\widehat H_{n,m_n}}_{\mathrm F}^4
  \right\} \\
 &=O_p(1).
\end{align*}
Here the empirical fourth moments centered at the empirical means are
$O_p(1)$ by the fourth-moment condition in \cref{ass:sketch} and the
inequality $\norm{\bm a-\bm b}_2^4\leq
8\{\norm{\bm a}_2^4+\norm{\bm b}_2^4\}$.
Therefore, for every $\varepsilon>0$,
\begin{equation}
  \frac1{m_n}\sum_{i=1}^{m_n}
  \norm{\widehat{\bm\psi}_{n,m_n,i}}_2^2
  \mathbf1\big\{\norm{\widehat{\bm\psi}_{n,m_n,i}}_2>\varepsilon\sqrt{m_n}\big\}
  \leq\frac1{\varepsilon^2m_n^2}\sum_{i=1}^{m_n}
  \norm{\widehat{\bm\psi}_{n,m_n,i}}_2^4\pto0.
  \label{eq:conditional-lindeberg}
\end{equation}

Consider any subsequence. By
\cref{eq:empirical-covariance,eq:conditional-lindeberg}, it has a further
subsequence along which both the conditional covariance convergence and the
conditional Lindeberg condition hold almost surely. The conditional
multivariate Lindeberg--Feller theorem then gives the Gaussian limit along
this further subsequence. The subsequence characterization of convergence in
probability therefore yields
\begin{equation}
  \sqrt{m_n}\,\widetilde\Delta_{n,m_n,1}^{*}
  \dto N(\bzero,\Omega)
  \quad\text{conditionally in $\PS$-probability}.
  \label{eq:bootstrap-moment-clt}
\end{equation}
Finally, \cref{prop:linearized-equivalence,eq:bootstrap-moment-clt} and the
conditional version of Slutsky's theorem prove \cref{eq:bootstrap-clt} for
the first bootstrap replicate.
\end{proof}

\section{\texorpdfstring{Proof of \Cref{cor:bootstrap-distribution-consistency}: bootstrap distribution consistency}{Proof of bootstrap distribution consistency}}
\label{app:distribution-proof}

\begin{proof}[Proof of \Cref{cor:bootstrap-distribution-consistency}]
Recall the distribution functions in
\cref{eq:ideal-distribution-functions,eq:empirical-bootstrap-cdf} and the
following definitions from \appref{app:proofs}:
\[
  \bm Z\sim N(\bzero,\Omega),
  \qquad
  G(t):=\Pp\{\norm{\bm Z}_2\leq t\}.
\]
Since $\rank(\Omega)\geq1$,
$\norm{\bm Z}_2$ is the square root of a nondegenerate weighted sum of
independent $\chi_1^2$ variables.  Hence $G$ is continuous on $\R$ and strictly increasing on $(0,\infty)$.

Applying the continuous mapping theorem to \cref{eq:target-clt}, followed by
P\'olya's theorem, gives
\[
  \sup_{t\in\R}|F_{n,m_n}(t)-G(t)|\longrightarrow0.
\]
For the conditional bootstrap distribution, the conditional continuous
mapping theorem applied to \cref{eq:bootstrap-clt} gives
$\widehat F_{n,m_n}(t)\pto G(t)$ for every $t$. Because $G$ is continuous,
the standard argument using a finite grid and monotonicity in the proof of P\'olya's
theorem applies in probability and gives
\[
  \sup_{t\in\R}|\widehat F_{n,m_n}(t)-G(t)|\pto0.
\]
The triangle inequality proves \cref{eq:ideal-uniform}.

Conditional on $\mathcal S_n$, the scaled bootstrap errors
$\sqrt{m_n}e_{n,m_n,1}^*,\ldots,\sqrt{m_n}e_{n,m_n,B_n}^*$ defined in
\cref{eq:exact-bootstrap-threshold} are i.i.d. with distribution function
$\widehat F_{n,m_n}$ from \cref{eq:ideal-distribution-functions}.  Thus the Dvoretzky--Kiefer--Wolfowitz inequality \citep{dvoretzky1956asymptotic,massart1990tight} gives, for every $\varepsilon>0$,
\begin{equation}
  \Pstar\left\{
  \sup_t|\widehat F_{n,m_n,B_n}(t)-\widehat F_{n,m_n}(t)|>\varepsilon
  \right\}
  \leq2\exp(-2B_n\varepsilon^2).
  \label{eq:dkw}
\end{equation}
If $B_n\to\infty$, taking expectations in \cref{eq:dkw} shows that the
Kolmogorov distance between the empirical bootstrap distribution and its
exact conditional distribution converges to zero in probability. Combining
this result with \cref{eq:ideal-uniform} proves \cref{eq:empirical-uniform}.
\end{proof}

\section{\texorpdfstring{Proof of \Cref{thm:exact-coverage}: asymptotic coverage}{Proof of asymptotic coverage}}
\label{app:coverage-proof}

\begin{proof}[Proof of \Cref{thm:exact-coverage}]
Recall that the empirical error bounds obtained by refitting and by
linearization are defined in
\cref{eq:exact-bootstrap-threshold,eq:linearized-error} with $m=m_n$ and
$B=B_n$. Recall the definitions
$\bm Z\sim N(\bzero,\Omega)$,
$G(t):=\Pp\{\norm{\bm Z}_2\leq t\}$, and
$q:=G^{-1}(1-\alpha)$ from \appref{app:proofs}.  The continuity and
strict monotonicity properties of $G$ are proved in
\appref{app:distribution-proof}.  The two convergence statements proved in
\appref{app:distribution-proof} imply
\[
  \sup_{t\in\R}|\widehat F_{n,m_n,B_n}(t)-G(t)|\pto0.
\]
Because $G$ is continuous and strictly increasing at $q$, quantile
consistency applies. Since $\sqrt{m_n}>0$, the empirical quantile of the scaled
bootstrap errors equals $\sqrt{m_n}$ times the empirical quantile of the
unscaled errors. Therefore,
\begin{equation}
  \sqrt{m_n}\,\widehat\epsilon_{n,m_n,B_n}(\alpha)\pto q.
  \label{eq:quantile-consistency}
\end{equation}
For every $\eta>0$,
\begin{align}
  &\Pp\!\left\{\sqrt{m_n}\norm{\bbetahat_{n,m_n}-\bbeta_n}_2\leq q-\eta\right\}
  -\Pp\!\left\{\left|\sqrt{m_n}\widehat\epsilon_{n,m_n,B_n}(\alpha)-q\right|>\eta\right\}\notag\\
  &\quad\leq
  \Pp\!\left\{\norm{\bbetahat_{n,m_n}-\bbeta_n}_2
  \leq\widehat\epsilon_{n,m_n,B_n}(\alpha)\right\}\notag\\
  &\quad\leq
  \Pp\!\left\{\sqrt{m_n}\norm{\bbetahat_{n,m_n}-\bbeta_n}_2\leq q+\eta\right\}
  +\Pp\!\left\{\left|\sqrt{m_n}\widehat\epsilon_{n,m_n,B_n}(\alpha)-q\right|>\eta\right\}.
  \label{eq:random-threshold-sandwich}
\end{align}
By \cref{eq:target-clt,eq:quantile-consistency}, taking lower and upper
limits in \cref{eq:random-threshold-sandwich} gives
\begin{align*}
  G(q-\eta)
  &\leq\liminf_{n\to\infty}
  \Pp\!\left\{\norm{\bbetahat_{n,m_n}-\bbeta_n}_2
  \leq\widehat\epsilon_{n,m_n,B_n}(\alpha)\right\}\\
  &\leq\limsup_{n\to\infty}
  \Pp\!\left\{\norm{\bbetahat_{n,m_n}-\bbeta_n}_2
  \leq\widehat\epsilon_{n,m_n,B_n}(\alpha)\right\}\\
  &\leq G(q+\eta).
\end{align*}
Letting $\eta\downarrow0$ and using continuity of $G$ and
$G(q)=1-\alpha$ proves \cref{eq:exact-coverage}. This argument does not
require the estimator and error bound to be independent.

For the linearized error bound, \cref{eq:bootstrap-moment-clt} and the
continuous mapping theorem show that the conditional distribution of
$\sqrt{m_n}\,\widetilde e_{n,m_n,1}^{*}$ converges to $G$ in
$\PS$-probability. The subsequence and P\'olya arguments used in
\appref{app:distribution-proof} make this convergence uniform in $t$.
Conditional on $\mathcal S_n$, the variables
$\widetilde e_{n,m_n,1}^{*},\ldots,\widetilde e_{n,m_n,B_n}^{*}$ are
independent and identically distributed. The Dvoretzky--Kiefer--Wolfowitz
inequality and $B_n\to\infty$ therefore give uniform convergence of their
empirical distribution function. The same positive-scaling property of
empirical quantiles then gives
\[
  \sqrt{m_n}\,\widetilde\epsilon_{n,m_n,B_n}(\alpha)\pto q.
\]
Replacing $\widehat\epsilon_{n,m_n,B_n}(\alpha)$ by
$\widetilde\epsilon_{n,m_n,B_n}(\alpha)$ in
\cref{eq:random-threshold-sandwich} and repeating the displayed
liminf--limsup argument yields
\[
  \Pp\!\left\{\norm{\bbetahat_{n,m_n}-\bbeta_n}_2
  \leq\widetilde\epsilon_{n,m_n,B_n}(\alpha)\right\}
  \longrightarrow1-\alpha,
\]
which is \cref{eq:linearized-coverage}.
\end{proof}

\section{\texorpdfstring{Proof of \Cref{prop:mc-protected}: Monte Carlo error control}{Proof of Monte Carlo error control}}
\label{app:mc-proof}

\begin{proof}[Proof of \Cref{prop:mc-protected}]
For the finite-sample result, fix the observed sketch. Using
$e_{n,m,1}^*$ from \cref{eq:exact-bootstrap-threshold} and the quantile
formula repeated in \appref{app:proofs}, define
\[
  F^*(t):=\Pstar\{e_{n,m,1}^*\leq t\},
  \qquad
  q^*:=F^{*-1}(p)=q_{n,m}^*(p),
  \qquad
  F^*(q^*-):=\lim_{t\uparrow q^*}F^*(t),
\]
\[
  N_B:=\sum_{b=1}^B\mathbf1\{e_{n,m,b}^*<q^*\},
  \qquad
  N_B\mid\mathcal S_{n,m}
  \sim\operatorname{Bin}\bigl(B,F^*(q^*-)\bigr).
\]
Since $F^*(q^*-)\leq p$, the conditional law of $N_B$ is stochastically
dominated by $\operatorname{Bin}(B,p)$.  Moreover,
$\{e_{n,m,(k)}^*\geq q^*\}=\{N_B\leq k-1\}$.  The definition of
$k_B(p,\delta)$ in \cref{eq:mc-order} therefore gives
\[
  \Pstar\{e_{n,m,(k_B(p,\delta))}^*\geq q^*\}\geq1-\delta
\]
when $k_B(p,\delta)<\infty$. If $k_B(p,\delta)=\infty$, the adjusted error
bound equals $+\infty$ by \cref{eq:mc-threshold}, so the result is immediate. This proves
\cref{eq:mc-conditional}.

For the asymptotic assertion, define
\[
  E_n:=\norm{\bbetahat_{n,m_n}-\bbeta_n}_2,
  \qquad
  q_n^*:=q_{n,m_n}^*(1-\alpha).
\]
Recall $G(t):=\Pp\{\norm{\bm Z}_2\leq t\}$ and
$q:=G^{-1}(1-\alpha)$ from \appref{app:proofs}; their continuity and
strict monotonicity properties are proved in
\appref{app:distribution-proof}. The uniform convergence result proved there
gives $\sup_t|\widehat F_{n,m_n}(t)-G(t)|\pto0$. Since $G$ is
continuous and strictly increasing at $q$, quantile consistency gives
\[
  \sqrt{m_n}\,q_n^*\pto q.
\]
For every $\eta>0$, the same sandwich argument used in
\appref{app:coverage-proof} gives
\begin{align*}
  &\Pp\!\left\{\sqrt{m_n}E_n\leq q-\eta\right\}
  -\Pp\!\left\{\left|\sqrt{m_n}q_n^*-q\right|>\eta\right\}
  \leq\Pp(E_n\leq q_n^*)\\
  &\qquad\leq
  \Pp\!\left\{\sqrt{m_n}E_n\leq q+\eta\right\}
  +\Pp\!\left\{\left|\sqrt{m_n}q_n^*-q\right|>\eta\right\}.
\end{align*}
By \cref{eq:target-clt}, first taking lower and upper limits and then
letting $\eta\downarrow0$ yields
$\Pp(E_n\leq q_n^*)\to1-\alpha$, and hence
$\Pp(E_n>q_n^*)\leq\alpha+o(1)$.  By
\cref{eq:mc-conditional}, for almost every sketch the conditional bootstrap
distribution satisfies
\[
  \Pstar\{\widehat\epsilon_{n,m_n,B_n}^{\mathrm{MC}}(\alpha,\delta)
  <q_n^*\}\leq\delta.
\]
Taking expectation with respect to the sketch distribution and applying the
union bound gives
\[
  \Pp\{E_n>\widehat\epsilon_{n,m_n,B_n}^{\mathrm{MC}}(\alpha,\delta)\}
  \leq\Pp(E_n>q_n^*)
  +\Pp\{\widehat\epsilon_{n,m_n,B_n}^{\mathrm{MC}}(\alpha,\delta)<q_n^*\}
  \leq\alpha+\delta+o(1).
\]
Taking complements and then the lower limit proves
\cref{eq:mc-coverage}. For each $n$, a finite adjusted error bound exists
exactly when $(1-\alpha)^{B_n}\leq\delta$.
\end{proof}

\section{\texorpdfstring{Proof of \Cref{thm:simultaneous-selection}: selection over a finite grid}{Proof of selection over a finite grid}}
\label{app:selection-proof}

\begin{proof}[Proof of \Cref{thm:simultaneous-selection}]
Recall the definitions of $\bbetahat_{n,j}$,
$\widehat\epsilon_{n,j}(\alpha/K)$, and $\widehat j_n$ from the setup
preceding \cref{thm:simultaneous-selection}, repeated in
\appref{app:proofs}. For $j=1,\ldots,K$, define the failure event for sketch
size $m_{n,j}$ by
\[
  \mathcal B_{n,j}
  :=\left\{\norm{\bbetahat_{n,j}-\bbeta_n}_2
  >\widehat\epsilon_{n,j}(\alpha/K)\right\}.
\]
The event $\mathcal B_{n,j}$ occurs when the estimated error bound for sketch
size $m_{n,j}$ is smaller than its actual coefficient error. If
$\widehat j_n=j$, this error bound is at most $\tau$. Hence, if the selected
estimate has coefficient error greater than $\tau$, then $\mathcal B_{n,j}$
must occur:
\[
  \left\{\widehat j_n=j,\
  \norm{\bbetahat_{n,j}-\bbeta_n}_2>\tau\right\}
  \subseteq\mathcal B_{n,j}.
\]
Therefore, the event in \cref{eq:familywise-coverage} is contained in
$\bigcup_{j=1}^K\mathcal B_{n,j}$. Since $K$ is fixed, Bonferroni's inequality
and the assumed marginal coverage bounds give
\begin{align*}
  &\limsup_{n\to\infty}
  \Pp\bigg(\bigcup_{j=1}^K
  \left\{\widehat j_n=j,\ \norm{\bbetahat_{n,j}-\bbeta_n}_2>\tau\right\}
  \bigg)\\
  &\quad\leq
  \limsup_{n\to\infty}\Pp\bigg(\bigcup_{j=1}^K\mathcal B_{n,j}\bigg)
  \leq\sum_{j=1}^K\limsup_{n\to\infty}\Pp(\mathcal B_{n,j})
  \leq\sum_{j=1}^K\frac{\alpha}{K}=\alpha.
\end{align*}
The argument does not require the sketches at different grid points to be
independent. This proves \cref{eq:familywise-coverage}.
\end{proof}

\section{Supplementary numerical experiments}
\label{sec:supp-experiments}

\subsection{Archived benchmark settings}
\label{sec:source-experiments}

This appendix records the available settings for historical experiments from
earlier versions. Four benchmark figures that appeared in those versions are
omitted because neither their original scripts nor their numerical output are
available. The real data used in those experiments are not included, and the
descriptions of the synthetic data do not determine every implementation
detail. Reconstructing the figures would therefore not constitute a verified
reproduction, so we do not use them to support any numerical claim.

The available record indicates that the omitted experiments used
YearPredictionMSD ($n=463715$, $d=90$), cpusmall ($n=8192$, $d=12$), and two
synthetic design matrices ($n=50000$, $d=100$) with condition numbers
$10^{12}$ and $10^2$. The other recorded settings were a dense Gaussian
sketch, $\alpha=0.05$, $B=20$, sketch ratios
$m/d\in\{5,10,15,20,25\}$, and unnormalized regularization parameters
$\rho\in\{0.01,0.1,1\}$. We report these settings only to explain the regularization-parameter
conversion in \cref{sec:normalization}. None of the numerical conclusions in
this paper depends on the omitted figures. 

\paragraph{Timing measurements and computing environment.}
The timing results reported in \cref{sec:computational-results} were obtained
on a MacBook Pro with an Apple M1 Max (10-core CPU) and 64 GB of memory,
running macOS 26.5.2 and Python 3.11.4. The environment used NumPy 2.3.5
and SciPy 1.17.0; NumPy was linked against Apple's Accelerate framework,
with BLAS-related thread counts fixed to one. The reported runtime ratios
compare replicate-generation time only and exclude the initial fit, which is
common to both methods. These timings depend on the hardware and
linear-algebra implementation and should be interpreted as measurements for
this reference environment rather than as hardware-independent performance
guarantees.

\subsection{Does coefficient error decrease at the inverse-square-root rate?}

The asymptotic theory scales coefficient error by $\sqrt m$. To examine
whether this scaling is reasonable at the tested sketch sizes,
\cref{fig:scaled-thresholds} plots the independent-sketch reference quantile
and each estimated error bound after multiplication by $\sqrt m$. A curve
should be approximately horizontal if the unscaled quantity is proportional
to $m^{-1/2}$.

The curves vary moderately across the tested sketch ratios. This is compatible
with an $m^{-1/2}$ error scale, but it is not a formal test of the asymptotic
theory. The error bound obtained with the order-statistic correction is larger
than the unadjusted bound because both use the same bootstrap errors obtained
by refitting, while the correction selects a higher order statistic.

\begin{figure}[!t]
\centering
\includegraphics[width=\linewidth]{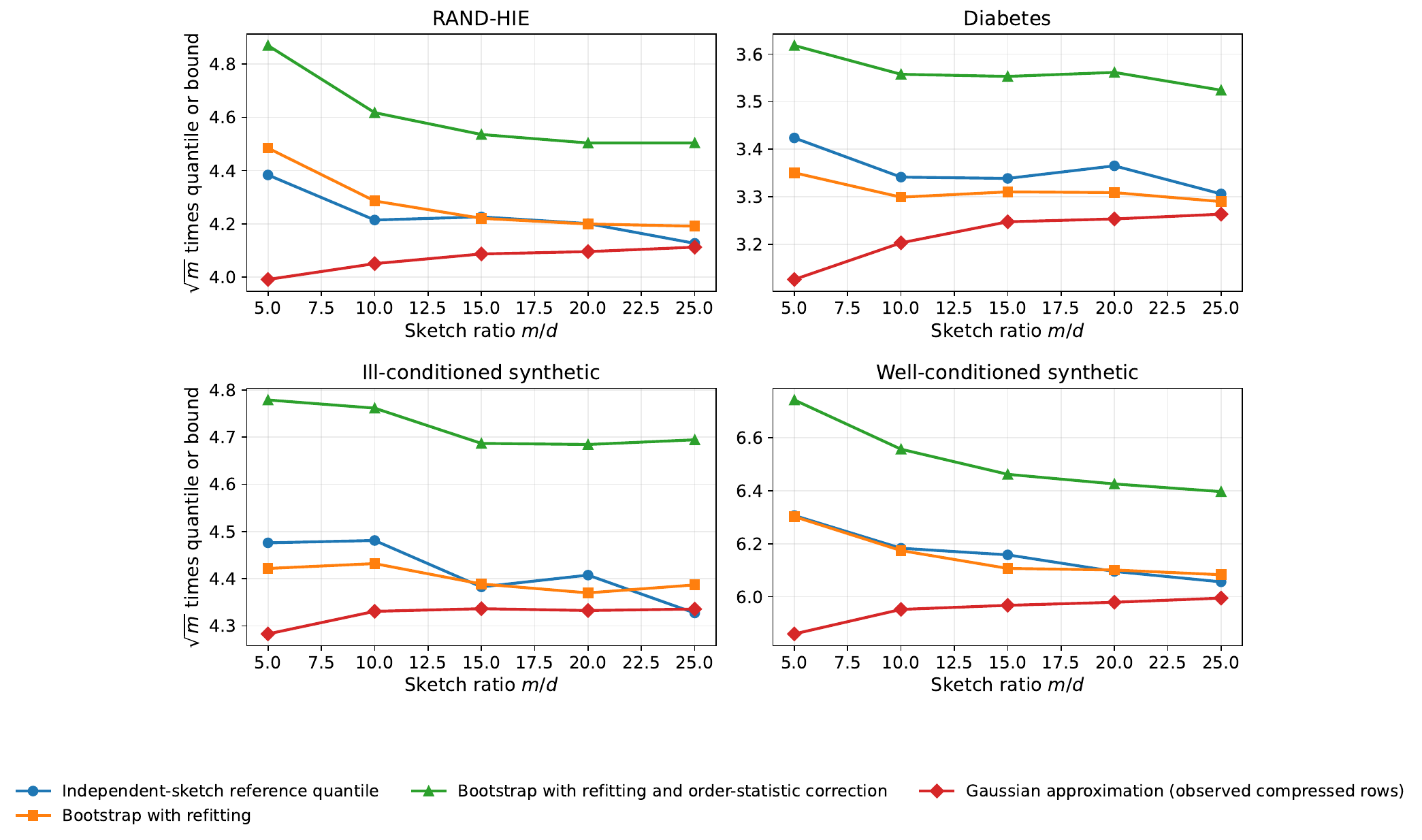}
\caption{Independent-sketch reference quantile and bootstrap error
bounds, each multiplied by $\sqrt m$. An approximately horizontal curve is
compatible with an $m^{-1/2}$ error scale.}
\label{fig:scaled-thresholds}
\end{figure}

\subsection{Calibration across nominal coverage levels}

This experiment uses four nominal coverage levels: $0.80$, $0.90$, $0.95$,
and $0.975$. It uses $m/d=15$, $B=199$, and bootstrap errors obtained by
refitting. The unadjusted and adjusted bounds use exactly the same
bootstrap errors and differ only in the selected order statistic. The diagonal
in \cref{fig:calibration} represents exact agreement between empirical and
nominal coverage.

At some lower nominal levels, the empirical coverage of the unadjusted bound
is below the nominal level. The adjusted bound generally gives coverage closer
to or above the nominal level. At nominal coverage $0.95$, the unadjusted
coverage is $0.940$ for RAND-HIE and $0.953$ for Ill-conditioned synthetic.
The adjusted coverage is $0.980$ for both data sets.

\begin{figure}[!t]
\centering
\includegraphics[width=0.94\linewidth]{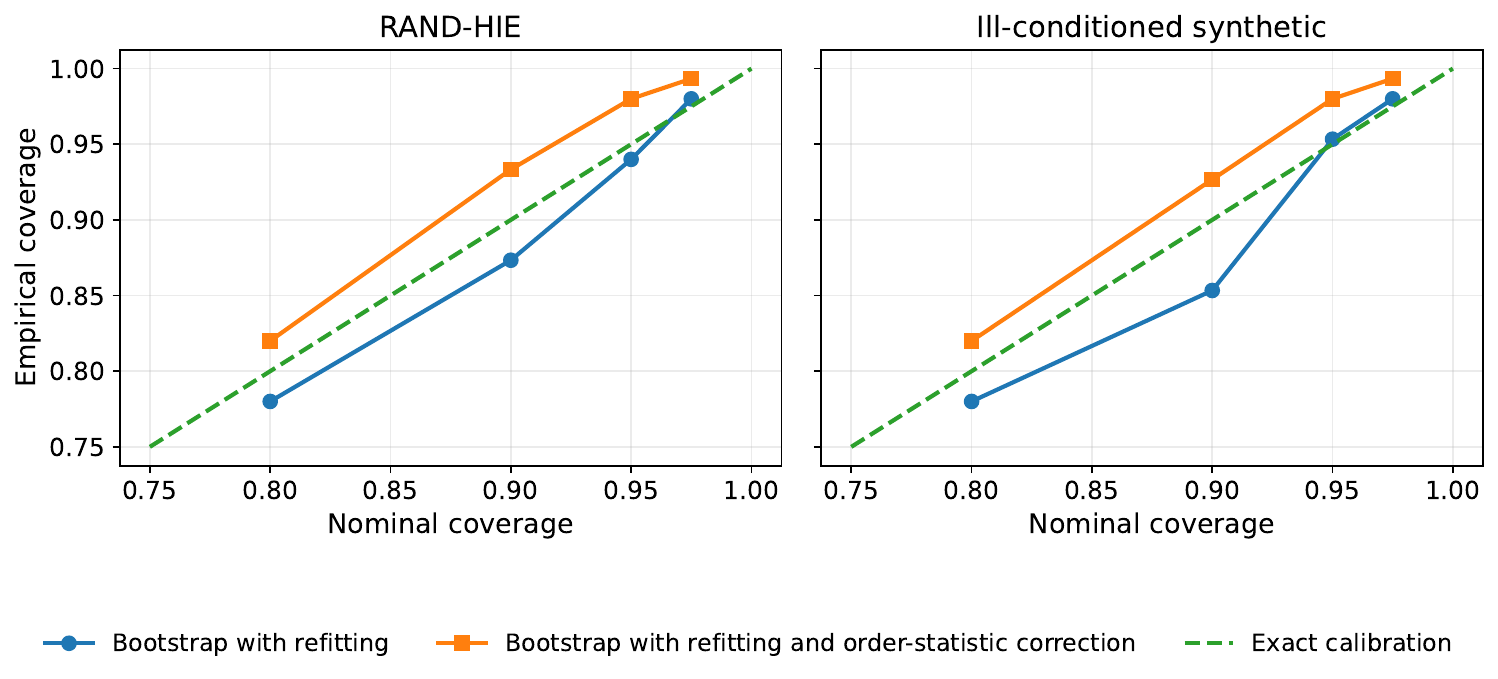}
\caption{Empirical coverage versus nominal coverage at $m/d=15$, $B=199$,
and $150$ Monte Carlo repetitions. Both bootstrap curves use the same
replicate errors obtained by refitting. The diagonal is exact calibration.}
\label{fig:calibration}
\end{figure}

\subsection{Sensitivity to regularization}
This sensitivity study was run independently of the main coverage
experiment in \cref{sec:coverage-results}, so the $\lambda=0.1$
entry below is based on a separate Monte Carlo run.

\Cref{fig:regularization,tab:regularization-summary} reports results for
$\lambda\in\{0,0.01,0.1,1\}$ on RAND-HIE at $m/d=15$. For each value, the
experiment uses $150$ Monte Carlo repetitions, $1500$ independent sketches to
estimate the independent-sketch reference quantile, and $5000$ Gaussian draws
to estimate each Gaussian-approximation quantile. We include $\lambda=0$ only
as a numerical check at the boundary; it lies outside the finite-sample
assumptions, which require $\lambda_n>0$. Because $\lambda$ affects the scale
of the coefficient error, the left panel reports coverage and the right panel
reports the mean error bound divided by the independent-sketch reference
quantile. Thus, the panels show coverage probability and relative bound size,
respectively.

Across the four values of $\lambda$, the mean unadjusted bound obtained by
refitting is $0.968$--$1.044$ times the independent-sketch reference quantile.
With the order-statistic correction, the ratio is $1.033$--$1.131$. Coverage
is not monotone in $\lambda$ in this finite-sample experiment, so the results
do not support a claim that stronger regularization always improves or worsens
calibration.

\begin{figure}[ht]
\centering
\includegraphics[width=0.96\linewidth]{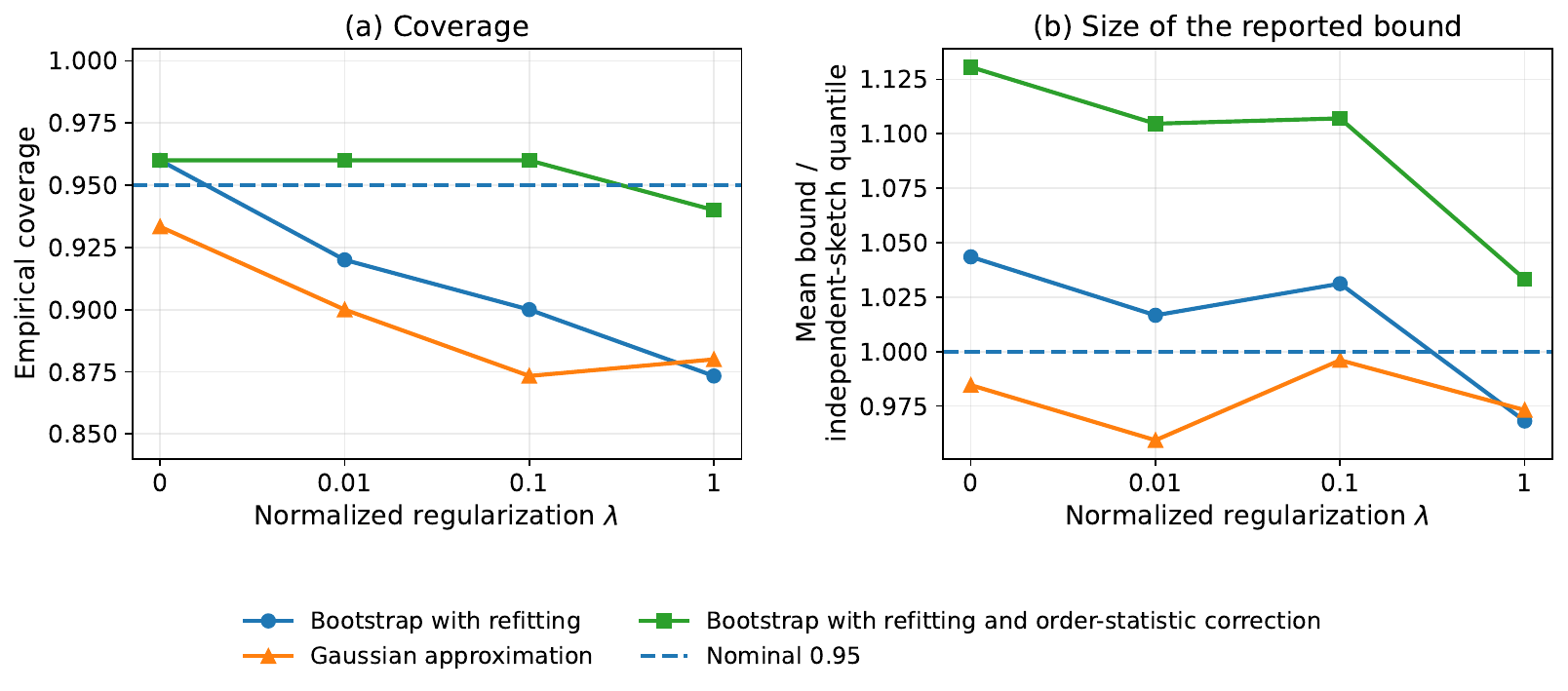}
\caption{Sensitivity to normalized regularization on RAND-HIE at $m/d=15$.
Left: empirical coverage. Right: mean reported error bound divided by the
independent-sketch reference quantile. Both bootstrap curves use errors
obtained by refitting.}
\label{fig:regularization}
\end{figure}

\begin{table}[t]
\centering
\caption{Regularization experiment on RAND-HIE at $m/d=15$. ``Unadjusted''
uses the empirical quantile of errors obtained by refitting; ``Adjusted''
applies the order-statistic correction to the same errors.}
\label{tab:regularization-summary}
\begingroup\small
\resizebox{\linewidth}{!}{%
\begin{tabular}{rrrrrrr}
\toprule
& \multicolumn{3}{c}{Empirical coverage} & \multicolumn{3}{c}{Mean bound / independent-sketch quantile} \\
\cmidrule(lr){2-4}\cmidrule(lr){5-7}
$\lambda$ & Unadjusted & Gaussian & Adjusted & Unadjusted & Gaussian & Adjusted \\
\midrule
0 & 0.960 & 0.933 & 0.960 & 1.044 & 0.985 & 1.131 \\
0.01 & 0.920 & 0.900 & 0.960 & 1.017 & 0.959 & 1.105 \\
0.1 & 0.900 & 0.873 & 0.960 & 1.031 & 0.996 & 1.107 \\
1 & 0.873 & 0.880 & 0.940 & 0.968 & 0.973 & 1.033 \\
\bottomrule
\end{tabular}}
\endgroup
\end{table}

\subsection{Sensitivity to the sketch distribution}

\Cref{fig:sketch-comparison,tab:sketch-summary} compares Gaussian projection,
Rademacher projection, and uniform row sampling on Diabetes for
$m/d\in\{5,10,15,20\}$. Each setting uses $120$ Monte Carlo repetitions and
$1200$ independent sketches to estimate the independent-sketch reference
quantile. For Gaussian and Rademacher projections, each
Gaussian-approximation quantile uses $3000$ Gaussian draws.
At $m/d=15$, the empirical coverage of the bootstrap with refitting is $0.933$,
$0.942$, and $0.933$, respectively. The corresponding ratios of the mean
error bound to the independent-sketch reference quantile are $1.003$, $0.991$,
and $0.995$. Thus, the bootstrap with refitting gives similar results for the
three sketch distributions in this experiment.

The small difference between the Gaussian and Rademacher projections does not
contradict \cref{cor:rademacher-optimal}. That result compares standardized
sketches with i.i.d. standardized entries and shows that the covariance of the leading
error term under the Rademacher sketch is no larger in the Loewner order. It
does not apply to uniform row sampling and does not imply that the finite-sample
difference between Gaussian and Rademacher projections must be large. We
evaluated the Gaussian approximation for the two projection sketches, but not
for uniform row sampling.

\begin{figure}[bth]
\centering
\includegraphics[width=\linewidth]{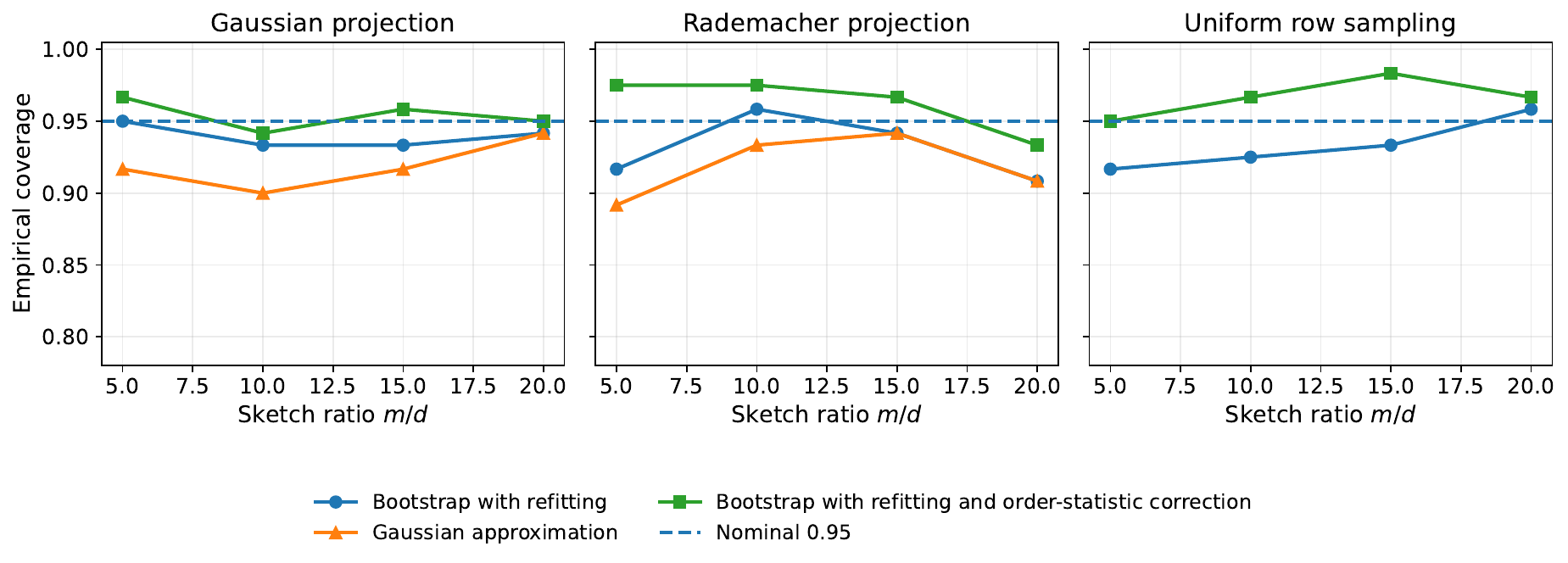}
\caption{Empirical coverage on Diabetes for three sketch distributions. Each
point uses $120$ Monte Carlo repetitions and $B=199$. Both bootstrap curves
use errors obtained by refitting. The Gaussian approximation was not evaluated
for uniform row sampling.}
\label{fig:sketch-comparison}
\end{figure}

\begin{table}[H]
\centering
\caption{Sketch-distribution experiment at $m/d=15$. ``Unadjusted'' uses the
empirical quantile of errors obtained by refitting; ``Adjusted'' applies the
order-statistic correction to the same errors.}
\label{tab:sketch-summary}
\begingroup\small
\resizebox{\linewidth}{!}{%
\begin{tabular}{lrrrrrr}
\toprule
& \multicolumn{3}{c}{Empirical coverage} & \multicolumn{3}{c}{Mean bound / independent-sketch quantile} \\
\cmidrule(lr){2-4}\cmidrule(lr){5-7}
Sketch & Unadjusted & Gaussian & Adjusted & Unadjusted & Gaussian & Adjusted \\
\midrule
Gaussian projection & 0.933 & 0.917 & 0.958 & 1.003 & 0.980 & 1.075 \\
Rademacher projection & 0.942 & 0.942 & 0.967 & 0.991 & 0.979 & 1.063 \\
Uniform row sampling & 0.933 & -- & 0.983 & 0.995 & -- & 1.063 \\
\bottomrule
\end{tabular}}
\endgroup
\end{table}

\subsection{Sensitivity to the number of bootstrap replicates}

The experiment in \cref{sec:finite-b-experiment} holds one sketch fixed and
uses $50{,}000$ bootstrap errors obtained by refitting to compute a
high-precision estimate of the conditional bootstrap quantile. It records how
often an estimate based on only $B$ replicates falls below this quantity. Here,
instead, we draw a new sketch in each Monte Carlo repetition and estimate
unconditional coverage as a function of $B$.

The experiment uses linearized bootstrap errors on RAND-HIE at $m/d=15$ with
$200$ Monte Carlo repetitions. In each repetition, it generates one sequence
of $B_{\max}=999$ errors. For every smaller value of $B$, it uses the first $B$
errors in that sequence. Thus, the unadjusted and adjusted bounds use exactly
the same errors and differ only in the selected order statistic.

The unadjusted coverage increases from $0.850$ at $B=20$ to $0.940$ at
$B=499$ and $999$. For $B=20$ or $49$, the required adjusted rank exceeds
$B$, so no finite adjusted error bound is available. When a finite adjusted
error bound is available, it gives higher coverage and a larger mean bound.
The difference between the adjusted and unadjusted bounds decreases as $B$
increases.

\begin{figure}[thb]
\centering
\includegraphics[width=0.82\linewidth]{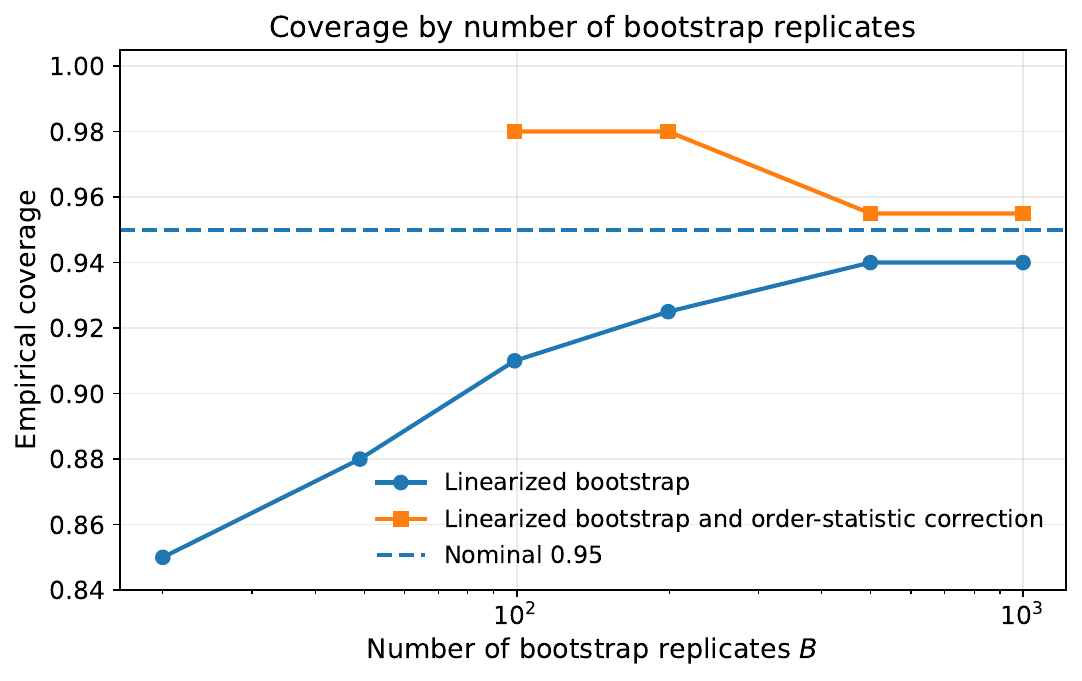}
\caption{Unconditional coverage as the number of bootstrap replicates changes
on RAND-HIE at $m/d=15$. Both curves use the same linearized bootstrap errors;
the order-statistic correction changes only the selected order statistic.}
\label{fig:bootstrap-size}
\end{figure}

\begin{table}[thb]
\centering
\caption{Unconditional coverage as $B$ changes, using linearized bootstrap
errors. A dash indicates that the required adjusted rank exceeds $B$, so no
finite adjusted error bound is available. ``Adjusted rank'' is the rank of the
selected error among the $B$ bootstrap errors. For every $B<B_{\max}$, the
calculation uses the first $B$ elements of the sequence with $B_{\max}=999$.}
\label{tab:bootstrap-size-summary}
\begingroup\small\setlength{\tabcolsep}{3pt}
\begin{tabular}{rccccc}
\toprule
$B$ & \shortstack{Unadjusted\\coverage} & \shortstack{Adjusted\\coverage} & \shortstack{Adjusted\\rank} & \shortstack{Mean unadjusted\\bound} & \shortstack{Mean adjusted\\bound} \\
\midrule
20 & 0.850 & -- & -- & 0.315 & -- \\
49 & 0.880 & -- & -- & 0.331 & -- \\
99 & 0.910 & 0.980 & 98 & 0.338 & 0.375 \\
199 & 0.925 & 0.980 & 195 & 0.338 & 0.363 \\
499 & 0.940 & 0.955 & 483 & 0.338 & 0.352 \\
999 & 0.940 & 0.955 & 961 & 0.337 & 0.346 \\
\bottomrule
\end{tabular}
\endgroup
\end{table}

\end{document}